\documentclass[12pt,times]{amsart}

\usepackage{amssymb, amsmath, setspace}
\usepackage[all,cmtip]{xy}
\usepackage{amsfonts}
\usepackage[myheadings]{fullpage}
\newtheorem{theorem}{Theorem}[section]
\newtheorem{prop}[theorem]{Proposition}

\newtheorem{lm}[theorem]{Lemma}

\newtheorem{conj}[theorem]{Conjecture}
\theoremstyle{definition}

\newtheorem{defn}[theorem]{Definition}
\newtheorem{rmk}[theorem]{Remark}
\newtheorem{eg}[theorem]{Example}

\usepackage{hyperref}

\newcommand{\ul} \underline

\newcommand{\mtx}{\left[ \begin{matrix}}
\newcommand{\mtxend}{\end{matrix}\right]}

\newcommand{\A}{\mathcal A}

\newcommand{\G}{\mathcal G}

\newcommand{\C}{\mathcal C}
\newcommand{\Z}{\mathcal Z}
\newcommand{\E}{\mathcal E}

\DeclareMathOperator{\<}{\prec}

\usepackage{color}
\definecolor{red}{rgb}{0.8,0,0}
\definecolor{green}{rgb}{0,0.8,0}

\let\oldmarginpar\marginpar
\renewcommand\marginpar[1]{\-\oldmarginpar[\raggedleft\footnotesize #1]%
{\raggedright\footnotesize #1}}

\begin{document}

\title{Algebraic Statistics for a Directed Random Graph Model with Reciprocation}\thanks{Supported in part by  NSF Grant DMS-0631589  to  Carnegie Mellon University.}

\author{Sonja Petrovi\'c}
\address{Department of Mathematics, Statistics, and Computer Science\\
University of Illinois, Chicago, IL, 60607}
\email{\tt petrovic@math.uic.edu}

\author{Alessandro Rinaldo}
\address{Department of Statistics\\
Carnegie Mellon University, Pittsburgh, PA 15213}
\email{\tt arinaldo@stat.cmu.edu}

\author{Stephen E. Fienberg}
\address{Department of Statistics and Machine Learning Department\\
 Carnegie Mellon University, Pittsburgh, PA 15213}
\email{\tt fienberg@stat.cmu.edu}

\date{}

\maketitle

\begin{abstract}
The $p_1$ model is a directed random graph model used to describe  dyadic interactions in a social network  in terms of  effects due to differential attraction (popularity) and expansiveness, as well as an additional effect due to reciprocation.  In this article we carry out an algebraic statistics analysis of this model. We show that the $p_1$ model is a toric model specified by a multi-homogeneous ideal. We conduct an extensive study of the Markov bases for $p_1$ models that incorporate explicitly the constraint arising from multi-homogeneity. We consider the properties of the corresponding  toric variety and relate them to the conditions for existence of the maximum likelihood and extended maximum likelihood estimator. Our results are directly relevant to the estimation and conditional goodness-of-fit testing problems in $p_1$ models. 
\end{abstract}

 \tableofcontents

\vskip0.2in
\section{Introduction}
\label{section:introduction}

The study of random graphs is an active area of research in many fields, including mathematics, probability, biology and social sciences. In the social network literature, the nodes of the network represent population units and the edges represent relationships among individuals. Thus, prescribing a probability distribution over a network is essentially a way of encoding the dynamic of interactions and behavior among individuals in a given population.  For a review, see Goldenberg et al.~\cite{GZAF:2009}.

One of the earlier and most popular statistical models for social networks is the $p_1$ model of Holland and Leinhardt, described in their 1981 seminal paper \cite{holl:lein:1981}. In a $p_1$ model 
the network is represented as a digraph in which, for every pair of nodes $(i,j)$, or dyad, one can observe one of four possible dyadic configurations: a directed edge from $i$ to $j$, a directed edge from $j$ to $i$, two directed edges and no edges between $i$ and $j$. Each dyad is in any of these configurations independently of all the other dyads. Quoting from \cite{holl:lein:1981}, the probability distribution over the entire network depends on various parameters quantifying the ``attractiveness" and "productivity'' of the individual nodes, the ``strength of reciprocity" of each dyad and the network ``density,"  or overall tendency to have edges. See the next section for details.   

Despite its simplicity and interpretability, and despite being a special case of the broader and well understood class of log-linear models, e.g., see~\cite{fien:wass:1981a,fien:meye:wass:1985}, the $p_1$ model poses extraordinary challenges of both theoretical and practical nature. Indeed, first and foremost, the number of network parameters depends on the size of the network, so that, as the population size grows, the model complexity also increases. This remarkable feature renders the $p_1$ model, as well as virtually any other network model, unlike traditional statistical models (including log-linear models), whose complexity is fixed and independent of the sample size.  Second, statistical inference for the $p_1$ model is typically made on the basis of {\it one} realization of the network, which amounts to one observation per dyad.   Because of the unavoidable sparsity and low information content of the observable data, and because of the dependence of the model complexity on the sample size, classical statistical procedures for estimation and goodness-of-fit testing in $p_1$ models exhibit properties that are poorly understood and are likely to be suboptimal or, at the very least, used inadequately.

In this article, we revisit the Holland-Leinhardt $p_1$ model using the tools and language  of algebraic geometry as is customary in the area of research known as algebraic statistics. See, e.g., \cite{diac:stur:1998}, \cite{pist:ricco:wynn:2001} and \cite{drto:stur:sull:2009}.
 For the class of $p_1$ models at hand, we identify the probabilities of each of the $2{n(n-1)}$  possible dyadic configurations with the indeterminates of a multi-homogeneous toric ideal. Our first goal is to investigate  the Markov bases for $p_1$ by computing a minimal generating set of this ideal first and then by removing basis elements that violate the condition of one observation per dyad. To our knowledge, this is the first time in the Markov bases literature that sampling constraints of this form, known in statistics as product Multinomial sampling schemes, have been directly incorporated in the study of Markov bases. Our results prescribe a set of rules for creating Markov moves that are applicable to network data and offer significant efficiency improvements over existing methods for generating Markov bases.  Our second goal is to study the toric variety associated to the $p_1$ ideal and to relate its geometric properties with the conditions for existence of the maximum likelihood estimator of the dyadic probabilities. 
 
The paper is organized as follows. In Section \ref{sec:p1.model} we describe in detail the $p_1$ model, its statistical properties and the statistical tasks of estimation and testing that motivate our work. In particular, we consider three different flavors of the $p_1$ model, of increasing complexity. In section \ref{sec:markov} we embark on an exhaustive study of the Markov bases for $p_1$ models. As we mentioned above, the 
 constraint that there is only one observation per 
 dyad   makes our analysis unusually difficult. In Section \ref{sec:mle} we describe the geometric properties of the $p_1$ model and we give explicit conditions for the existence of the MLE.

\section{The Holland-Leindhardt  $p_1$ Model}\label{sec:p1.model}

We consider a directed graph on the set of $n$ nodes. The nodes correspond to units in a network, such as individuals, and the edges correspond to links
between the units.  We focus on dyadic pairings and keep track of whether node $i$ sends an edge to $j$, or vice versa, or none, or both.   Let $p_{i,j}(1,0)$ be the probability of $i$ sending an edge toward $j$, and let $p_{j,i}(0,1)$ the probability of $j$ sending an edge toward $i$ (thus, $1$ denotes the outgoing side of the edge).  Further,  
$p_{ij}(1,1)$ is the probability of $i$ sending an edge to $j$ {\emph{and}}  $j$ sending an edge to $i$, while $p_{i,j}(0,0)$ is the probability that there is no edge between $i$ and $j$. Thus, 
\begin{equation}\label{eq:sum.1}
p_{i,j}(0,0) + p_{i,j}(1,0) + p_{i,j}(0,1) + p_{i,j}(1,1) = 1,
\end{equation}
for each of the ${n \choose 2}$ pairs of nodes $(i,j)$.
%
The Holland-Leinhardt  $p_1$ model of interest is given as follows (see \cite{holl:lein:1981}):
\begin{equation}\label{eq:hl.eqs}
\begin{array}{rcl} 
    \log p_{ij}(0,0) &= &\lambda _{ij}\\
    \log p_{ij}(1,0) &= & \lambda _{ij} + \alpha_i + \beta_j + \theta \\ 
    \log p_{ij}(0,1) &= & \lambda _{ij} + \alpha_j + \beta_i + \theta \\
    \log p_{ij}(1,1) &= & \lambda _{ij} + \alpha_i + \beta_j + \alpha_j + \beta_i + 2\theta + \rho_{i,j}.
\end{array}
\end{equation}
The real-valued variables $\theta$, $\alpha_i$, $\beta_i$, $\rho_{i,j}$ and $\lambda_{i,j}$ for all $i< j$ are the {\it model parameters}. The parameter $\alpha_i$ controls the effect of an outgoing edge from $i$, the parameter $\beta_j$ the effect of an incoming edge into $j$, and  $\rho_{i,j}$ the added effect of reciprocated edges (in both directions). The ``density" parameter $\theta$ quantifies the overall propensity of the network to have edges. Finally, $\lambda_{ij}$ is just a normalizing constant to ensure that \eqref{eq:sum.1} holds for each each dyad $(i,j)$. See the original paper on $p_1$ model by \cite{holl:lein:1981} for an extensive interpretation of the model parameters.

We take note here of a basic yet rather fundamental feature of our settings that apply to $p_1$ models as well as to many other statistical models for network: data become available in the form of {\it one} observed network. Even though for each pair of nodes $(i,j)$, the four probabilities $p_{i,j}(\bullet,\bullet)$ are strictly positive according to the defining equations \eqref{eq:hl.eqs}, {\it each dyad can be observed in one and only one of the 4 possible states.} Thus, despite the richness and complexity of $p_1$ models, their statistical analysis is severely limited by the disproportionally small information content in the data. This  fact makes $p_1$ and, more generally, network models challenging, and it greatly affects our analysis, as we will show.

We study the following special cases of the general $p_1$ structure:
\begin{enumerate}
  \item  $\rho_{ij}=0$, {\it no reciprocal effect}.
    \item $\rho_{ij}=\rho$, {\it constant reciprocation}.
    \item $\rho_{ij}=\rho + \rho_i+\rho_j$, {\it edge-dependent reciprocation}.
\end{enumerate}


For a network on $n$ nodes, we represent the vector of $2n(n-1)$ dyad probabilities as
\[
p=(p_{1,2},p_{1,3},\ldots,p_{n-1,n}) \in \mathbb{R}^{2n(n-1)},
\]
where, for each $i <j$, $p_{i,j} = \left( p_{i,j}(0,0), p_{i,j}(1,0) ,p_{i,j}(0,1),p_{i,j}(1,1) \right) \in \Delta_3$, with $\Delta_3$ the standard simplex in $\mathbb{R}^4$.

As usual in algebraic statistics, the $p_1$ \emph{model} is defined to be the set of all candidate probability distributions that satisfy the Holland-Leinhardt equations \eqref{eq:hl.eqs}.
By definition, the $p_1$ model is log-linear; that is, the set of candidate probabilities $p$ is such that $\log p$ is in the linear space spanned by the rows of a
 matrix $A$, which is often called the \emph{design matrix of the model}. 
  Indeed,  the design matrix encodes a homomorphism between
 two polynomial rings:
\begin{align*}
	\varphi_n: \mathbb C[p_{ij}(a,b) : i<j\in\{1\dots n\},a,b\in\{0,1\}]\to \mathbb C[\lambda_{ij},\alpha_i,\beta_i,\theta,\rho_{ij} : i<j\in\{1,\dots,n\}].
\end{align*}
 induced by 
\begin{align*}
	p_{ij}(a,b) \mapsto \lambda_{ij} \alpha_i^{a} \alpha_j^{b} \beta_i^{b} \beta_j^{a} \theta^{a+b} \rho_{ij}^{\min(a,b)}
\end{align*}
where $a,b\in\{0,1\}$, and we consider parameters  
$\lambda_{ij}$, $\alpha_i$, $\beta_i$, $\rho_{ij}$ and $\theta$ for $i,j\in\{1,\dots,n\}$
as unknowns.
The design matrix $A$ simply describes the action of $\varphi_n$ on the probabilities $p_{ij}(\bullet,\bullet)$.  The columns of
$A$ are indexed by $p_{ij}$'s and its rows by the model parameters. The entries of 
the design matrix are either $0$ or $1$; there is a $1$  
in the $(r,c)$-entry of the matrix if the parameter indexing row $r$ appears in the image of the $p_{ij}$ indexing the column $c$.  
For example, in the case $\rho_{ij}=0$, the matrix $A$ is of size $(2n)\times (3{{n}\choose {2}})$.  
We will create the design matrices 
in a systematic way:
the rows will always be indexed by 
$\lambda_{ij},\alpha_1,\dots,\alpha_n$, $\beta_1,\dots,\beta_n$, $\theta$, $\rho_{ij}$ 
lexicographically in that order. The columns will be ordered in the following way:
first fix $i$ and $j$ in the natural lexicographic ordering; then, within each set, vary the edge directions in this order: $(0,0)$, $(1,0)$, $(0,1)$, $(1,1)$. Examples can be found in section \ref{section:small-Markov}. 
Furthermore, it is easy to see that the design matrix for the network on $n$ nodes will consist of several copies of the $2$-node matrix, placed as a submatrix of in those rows and columns corresponding to all $2$-node subnetworks.

Let $\zeta$ denote the vector of model parameters, whose dimension $d$ depends on the type of restrictions on the $p_1$. Then, using the design matrix $A$, one can readily verify that the Holland-Leinhardt equations have  the log-linear representation 
\[
\log p = A^\top \zeta,
\]
or, equivalently, letting $a_k$ and $p_k$ be the $k$-th column of $A$ and the $k$-th entry of $p$, respectively,
\[
p_k = \prod_{l=1}^d a_k(l)^{\zeta(l)}.
\]

\subsection{Algebraic Statistical Challenges of $p_1$ Models}

We now describe two fundamental statistical tasks for the analysis of $p_1$ models that, both in theory and practice, still remain in part unsolved.

Denote by $M_A$ the $p_1$ model, that is, the set of points satisfying the Holland-Leinhardt equations \eqref{eq:hl.eqs}. Notice that $M_A$ is a relatively open set in the positive orthant of $\mathbb{R}^{2n(n-1)}$ of dimension  $\mathrm{rank}(A)$.  Let $\mathcal{X}_n \subset \mathbb{R}^{2n(n-1)}$ be the set of all observable networks on $n$ nodes, \textit{the sample space.} Notice that $|\mathcal{X}_n| = 4^{n(n-1)}$.  
We will write every point $x$ in the sample space $\mathcal X$ as
\[
x = (x_{1,2},x_{1,3},\ldots,x_{n-1,n}), 
\]
where each of the  ${n \choose 2}$ subvectors $x_{i,j}$ is a vertex of $\Delta_3$. 

Probabilistically, each $x_{i,j}$ is the realization of a random vector in $\mathbb{R}^4$ having Multinomial distribution with size $1$ and class probabilities $p^0_{i,j} \in \Delta_3$, where
\[
p^0 = (p^0_{1,2},p^0_{1,3},\ldots,p^0_{n-1,n})
\]
is an unknown vector in $M_A$ (thus, $p^0 > 0)$. Furthermore, \eqref{eq:hl.eqs} implies the Multinomial vectors representing the dyad configurations are mutually independent.

Once a network $x \in \mathcal{X}_n$ has been observed, statisticians are interested in the following interrelated fundamental problems:
\begin{itemize}
\item estimation problem: to estimate the unknown vector $p^0 \in M_A$;
\item goodness-of-fit problem: to test whether the $p_1$ model $M_A$ can be considered appropriate, i.e. whether the estimate produced in part (a) is a good fit to the observed data $x$.
\end{itemize}

To tackle the first problem, the most commonly used method is maximum likelihood estimation, which we now describe. For a fixed point $x \in \mathcal{X}_n$, the {\it likelihood function} $\ell \colon M_A \rightarrow [0,1]$ given by
\[
\ell_x(p) = \prod_{i<j} \left( \prod_{k=1}^4 p_{i,j}(k)^{x_{i,j}(k)}\right)
\]
returns the probability of observing the network $x$ as a function of the Multinomial probabilities $p \in M_A$. The {\it maximum likelihood estimate}, or MLE, of $p^0$ is 
\begin{equation}\label{eq:MLE}
\hat{p} = \mathrm{argmax}_{p \in M_A} \ell_x(p),
\end{equation}
i.e. the vector of Multinomial probabilities in $M_A$ that makes $x$ the most likely network to have been observed. Despite $\ell_x$ being smooth and convex for each $x \in \mathcal{X}_n$, it may happen that for some $x$, the supremum of the $\ell_x$ is achieved on the boundary of $M_A$. In this case, $\hat{p}$ does not belong to $M_A$, since it will have some zero coordinates. In this case, the MLE is said to be non-existent. Indeed, if a vector $p$ with zero entries is to satisfy the Holland and Leinhardt equations, then some of the model parameters must be equal to $-\infty$.

While the estimation problem is relatively well constrained, the problem of testing for goodness of fit has a much broader scope, as it entails testing whether the assumption that $p^0 \in M_A$ is itself appropriate. If it turns out, based on available data, that this assumption is rejected, than the next step is to further test that a different, typically more complex, $p_1$ model $M_{A'} \supset M_A$ is appropriate. In extreme cases, one may be led to conclude that no  $p_1$ model is adequate given the observed network. 
A typical goodness-of-fit testing proceeds through the following steps:

\begin{enumerate}
\item Compute the MLE $\hat{p}$ as defined in \eqref{eq:MLE}.
\item Compute the goodness-of-fit statistic $GF(x)$. This quantity measures how close the MLE is to the  observed network, or, using a statistical jargon, ``how well the model $M_A$ fits the data $x$." Among the most commonly used statistics are Pearson's $\chi^2$ and the likelihood ratio statistic:
\[
\sum_{i<k}\sum_{k=1}^4 \frac{\left( \hat{p}_{i,j}(k) - x_{i,j}(k) \right)^2}{\hat{p}^2_{i,j}(k)} \quad \text{and} \quad \sum_{i<k}\sum_{k=1}^4 x_{i,j}(k) \log \left( \frac{x_{i,j}(k)}{\hat{p}_{i,j}(k)} \right),
\]
respectively.
\item Reject the assumption that $p^0 \in M_A$ if the goodness-of-fit statistic used in step (2) is ``too large."
\end{enumerate}

Step (3) above is clearly the most delicate, as there is no generally valid recipe for deciding when $GF(x)$ is too large, one of the reasons being that the goodness-of-fit statistic is itself a random variable. In many applications, it is customary to declare a model a \emph{poor fit} if $G(x)$ is greater than a given deterministic threshold (depending on the size of the network and the number $d$ of model parameters), which is obtained based on asymptotic approximations. Despite their widespread use, for the present problem these approximations have no theoretical justification (see, e.g., \cite{haberman:81}) and can, in fact, be very poor. 

Markov bases provide an alternative, non-asymptitic approach to performing goodness-of-fit tests and model selection. Their use has gained some momentum and popularity in recent years.
Let $t = Ax$ denote the vector of margins, or {\it sufficient statistics}, corresponding to the observed network. Let $\mathcal{T}_t = \{ x' \in \mathcal{X}_n \colon A x' = t\} \subset \mathcal{X}_n$ be the fiber of $t$. From statistical theory (see, e.g., \cite{bfh:07}), all the networks in $\mathcal{X}_n$ belonging to the same fiber will produce the same MLE. Thus, from the inferential standpoint, all the networks in the same fiber are equivalent. Based on this, if $x$ is the observed network and $t$ the associated sufficient statistic, one can decide whether the model $M_A$ provides a good fit if the quantity
\begin{equation}\label{eq:alpha}
\alpha_x = \frac{|\{x' \in \mathcal{T}_t \colon GF(x') > GF(x) \}|}{|\mathcal{T}_t|}
\end{equation}
is large. Notice that $\alpha_x$ is the fraction of networks in the fiber at $t$ whose goodness-of-fit statistic value is larger than the value at the true network $x$. Heuristically, if $\alpha_x$ is large, i.e. closer to $1$, then  $\hat{p}$ is closer to the observed network than most of the other points in the same fiber, thus implying that the model fits really well.

Despite its appeal, this methodology is rarely feasible due to the high computational costs of determining the fiber.
Thus, rather than computing the fiber exactly, one can attempt to estimate $\alpha_x$ by performing a sufficiently long random walk over $\mathcal{F}_t$, as follows. A Markov basis for $M_A$ is set of vectors  $\mathcal{M} = \{f_1,\ldots,f_M \} \subset \mathbb{Z}^{2n(n-1)}$ such that
\begin{enumerate}
\item[(i)] $\mathcal{M} \subset \mathrm{kernel}(A)$;
\item[(ii)] for each observable margin $t$ and each $x,x' \in \mathcal{T}_t$, there exists a sequence $(\epsilon_1,f_{i_1}), \ldots, (\epsilon_N,f_{i_N})$ (with $N$ depending on both $x$ and $x'$), such that $\epsilon_j \in \{-1,1\}$ for all $j$ and
\[
x' = x + \sum_{j=1}^N \epsilon_j f_{i_j}, \quad x + \sum_{j=1}^a \epsilon_j f_{i_j} \in \mathcal{X}_n \; \text{ for all } \; a=1,\ldots,N.
\]
\end{enumerate}
Notice that Markov bases are not unique. If a Markov basis is available, then, for any observable network $x$ with sufficient statistics $t = Ax$, it is possible to estimate $\alpha_x$ by walking at random inside the fiber $\mathcal{T}_t$ according to the following algorithm: 
\begin{enumerate}
\item set k=0 and set $x^{\mathrm{old}} = x$;
\item at every step, starting from the current position $x^{\mathrm{old}}$,
\begin{enumerate}
\item randomly choose a vector $f \in \mathcal{M}$ and a number $\epsilon \in \{-1,1\}$;
\item   if $x^{\mathrm{old}} + \epsilon f \in \mathcal{X}_n$, move to $x^{\mathrm{new}} = x^{\mathrm{old}} + \epsilon f \in \mathcal{X}_n$, otherwise stay at $x^{\mathrm{new}} = x^{\mathrm{old}}$;
\end{enumerate}
\item if $GF(x^{\mathrm{new}})$ is larger than $GF(x)$ set k = k+1;
\item repeat the steps (2) and (3)  $K$ times.
\end{enumerate}

Provided that the algorithm is run for a sufficiently long time (i.e. $K$ is large enough), the number $\frac{k}{K}$ is an accurate estimate of $\alpha_x$. See \cite{diac:stur:1998} for details and variants of the basic Markov basis algorithm described above.

The goal of this article is to investigate the algebraic and geometric properties of $p_1$ models that are relevant to the estimation and goodness-of-fit testing problems. 
In Section \ref{sec:markov}, we study the toric ideal $I_A$ to construct Markov bases for $p_1$ models that are applicable to points belonging to the sample space $\mathcal{X}_n$. 
From the geometric side, in Section \ref{sec:mle} we turn our attention to the (closure of the) set $M_A$, which is a subset of the toric variety of $I_A$, and show its relevance and for the maximum likelihood estimation problem.

\section{Markov Bases of $p_1$ Models}
\label{sec:markov}
 
In this section we study the properties of Markov bases for the three versions of the $p_1$ model described in section \ref{sec:p1.model}.

Our analysis presents certain complications that sets it apart from most of the existing literature on Markov bases.
Indeed, the traditional algebraic geometry machinery generates Markov bases that are ``universal", in the sense of depending only on the design matrix $A$, and not the sample space and its properties. The reason for this is that Markov bases are obtained as generating sets toric ideals, and thus they cut out a complex toric variety; while the model itself lies in the positive real part of the variety. As a result, Markov bases tend to be very large even for design matrices of moderate size. In contrast, as we noted above, the sample space for network data is very highly constrained, since each dyad can only be observed in one and only one of the four possible configurations. Consequently, many of the basis elements of any Markov bases are not applicable, since they violate this basic constraint of the data.  Thus, once we find the Markov bases, we still need to be careful in identifying what elements are relevant to our data and in removing the moves that are not applicable.
To our knowledge, this is the first time this issue has been addressed in the Markov bases literature.

On the other hand, we are able to decompose every Markov basis element using certain ``basic" moves (Theorem \ref{thm:essentialMarkovForModel}), which are, as we will see, statistically meaningful by definition.  
The key idea is to decompose the toric ideal of the $p_1$ model using ideals which are known and easier to understand.
Namely, ignoring the normalizing constants $\lambda_{ij}$ reveals a connection between $p_1$ models and toric varieties which are associated to certain graphs.  
These have been studied by the commutative algebra and algebraic geometry community, specifically Villarreal \cite{Vill} and Ohsugi and Hibi \cite{OH}. 
We will use this connection to explain the form of Markov bases of $p_1$ models.  In terms of ideal generators for the $p_1$ model, 
re-introducing the normalizing constants does add another level of difficulty. 
 However, in terms of moves on the network, we can avoid this difficulty by exhibiting the decomposition of the moves (although inapplicable in terms of ideal generators) using well-understood binomials arising from graphs.  
This approach reduces the complexity and size of Markov moves. In addition, it allows us to bypass the study of those basis elements are not applicable due to the constraints described above.

\subsection{The toric ideal of the $p_1$ model}
\label{section:algebra-toric-ideal-of-p1}

Recall that the model consists of the points of the probability simplex that are in the row space of the design matrix. It follows that our model is, in fact, the
positive part of the variety that is (the closure of) the image of the map $\varphi_n$. 
For more discussion on the geometry of the model, see Section \ref{sec:mle}.  
To understand the model, we ask for all the polynomial equations that vanish on all of the points of the model; this set of equations is the defining ideal of the variety.
In the case of log-linear models, the ideal is a \emph{toric ideal} (see \cite{St} for a standard reference). It can be computed from the kernel of the design matrix $M$:
\begin{align*}
	I_M = \left( p^u-p^v : u-v \in \mathrm{kernel}(M) \right),
\end{align*}
where $p^u$ denotes a monomial. 
Any \emph{generating set} of the defining ideal $I_M$ gives a Markov basis for our model. This is the Fundamental Theorem of Markov Bases (see \cite{diac:stur:1998},
or \cite{drto:stur:sull:2009}, Theorem 1.3.6). 
It describes the relations among the $p_{ij}(\bullet,\bullet)$,
and can be used for a random walk in any given fiber consisting of the points with the same minimal sufficient statistics, and thus to compute the exact distribution of the model for goodness-of-fit purposes. Note that the sufficient statistics, in this case, are the in- and out- degree distributions of the nodes.  

In order to enumerate \emph{all} networks with the same degree distributions, one might want to use a \emph{Gr\"obner basis} instead.  A Gr\"obner basis is a generating set of the ideal, usually non-minimal, with some special structure desirable to have for computations. It is guaranteed to connect all points in a given fiber; every Gr\"obner basis is a Markov basis.  A Gr\"obner basis can be made unique by requiring it to be \emph{reduced}.  There are finitely many reduced Gr\"obner bases, and the union of all of them  is contained in the set of primitive binomials, called the \emph{Graver basis} of the ideal.  The Graver basis is generally very hard to compute, but sometimes easier to describe algebraically, and its structure naturally implies constraints on the structure of the minimal Markov moves.  

Our first goal is to understand the structure of these Markov bases for the three cases of the $p_1$ model.  Even though their size grows rapidly as we increase the number of nodes, there is quite a lot of structure in these generating sets.  In what follows, we will first illustrate this structure on some small networks. 

Let us first fix some notation for the remainder of the paper.
Since there are three cases of the $p_1$ model, we need three different names for the design matrices of the $n$-node network.  The design matrix depends on the choice of $n$ and $\rho_{ij}$:
\begin{enumerate}
    \item  For the case $\rho_{ij}=0$, when the reciprocal effect is \emph{zero}, the design matrix for the $n$-node network will be denoted by $\Z_n$.
    \item For the case of  \emph{constant} reciprocation, i.e. $\rho_{ij}=\rho$, the $n$-node network matrix will be denoted by $\C_n$.
    \item When reciprocation is \emph{edge-dependent}, i.e. $\rho_{ij}=\rho + \rho_i+\rho_j$, the design matrix will be denoted by $\E_n$.
\end{enumerate}

\smallskip

\subsection{Markov bases of small networks}
\label{section:small-Markov}

\subsubsection{Case I: no reciprocation} {\bf{($\rho_{ij}=0$})}
\label{subsec:smallMarkov-noRho}

This is clearly a special case of $\rho_{ij}=\rho$, but we treat it here separately.  We will see that, algebraically, it is interesting in its own right.

Let's start with the simplest nontrivial example: $n=2$.  The design matrix $\Z_2$ which encodes the parametrization $\varphi_2$ of the variety is
\begin{align*}
    \Z_2 = \mtx
    1&1&1&1\\
    0&1&0&1\\
    0&0&1&1\\
    0&0&1&1\\
    0&1&0&1\\
    0&1&1&2
    \mtxend 
    \begin{matrix}
	    \lambda_{12}\\ \alpha_1 \\ \alpha_2 \\ \beta_1 \\ \beta_2 \\ \theta
    \end{matrix} 
\end{align*}
The rows of $\Z_2$ are indexed by parameters as indicated, while the columns are indexed by $p_{12}(0,0)$, $p_{12}(1,0)$, $p_{12}(0,1)$, $p_{12}(1,1)$.
The ideal $I_{\Z_2}$ is the principal ideal generated by one quadric:
\begin{align*}
	I_{\Z_2} = ( p_{12}(1,0) p_{12}(0,1) - p_{12}(1,1)p_{12}(0,0) )
\end{align*}
and thus this single binomial is a Markov basis, and also a Gr\"obner basis with respect to any term order.
This can be verified by hand, or using software such as {\tt 4ti2} \cite{4ti2}.
\begin{rmk}[{\bf From binomials to moves}]
	In general, we can translate binomials to moves in the following way: 
	we will remove all edges that are represented by the $p_{ij}$'s in the negative monomial, and add all edges represented by the $p_{ij}$'s in the positive 		monomial.  Note that if $p_{ij}(0,0)$ occurs in either, it has no effect: it says to remove or add the ''no-edge'', so nothing needs to be done. 
	However, there is a reason why 		$p_{ij}(0,0)$'s show up in the ideal: the structure of $\Z_n$ for any $n$ requires that each binomial in the ideal is homogeneous with respect to the pair $\{i,j\}$.  Here, for example, since the positive monomial is of degree two, the negative 		monomial has $p_{12}(0,0)$ attached to it to ensure it also is of degree two.\\	
	Thus, the generator of $I_{\Z_2}$ represents the following Markov move:  
	\begin{verbatim} 
		Delete the bidirected edge between 1 and 2, 
		and replace it by an edge from		1 to 2 and an edge from 2 to 1.
	\end{verbatim}
\end{rmk}
However, if we would like to require at most one edge per dyad, then this binomial is meaningless and there aren't really any allowable Markov moves.  
Philosophically, the case of no reciprocation somehow contradicts this requirement, since if $\rho_{ij}=0$, we always get that a bi-directed edge between two nodes is valued the same as two edges between them.  Thus the assumption of only one edge per dyad makes this problem so much more complicated, as relations like this one for any dyad in an $n$-node network will appear in the generating sets of the ideal $I_{\Z_n}$, but we will \emph{never} want to use them.

Next, let $n=3$.  The toric ideal $I_{\Z_3}$ is minimally generated by the following set of binomials:
\begin{align*}
    p_{23}(0,1) p_{23}(1,0) &-p_{23}(1,1)p_{23}(0,0),\\
    p_{13}(0,1) p_{13}(1,0) &-p_{13}(1,1)p_{13}(0,0),\\
    p_{12}(0,1) p_{12}(1,0) &-p_{12}(1,1)p_{12}(0,0),\\
    p_{12}(0,1) p_{13}(1,0) p_{23}(0,1) & -p_{12}(1,0) p_{13}(0,1) p_{23}(1,0).
\end{align*}
It is interesting to note that the first $3$ generators are precisely the binomials from $I_{\Z_2}$ for the three dyads $\{1,2\}$, $\{1,3\}$, and $\{2,3\}$.
The only statistically meaningful generator is the cubic. It represents the following move:
\begin{verbatim}
	Replace the edge from 1 to 2 by the edge from 2 to 1;
	replace the edge from 2 to 3 by the edge from 3 to 2;
	replace the edge from 3 to 1 by the edge from 1 to 3.
\end{verbatim}
Graphically, it represents the three-cycle oriented two different ways: the positive monomial represents the cycle $1\to 3\to 2\to 1$, while the negative monomial represents the cycle $1\to 2\to 3\to 1$.

Suppose now that $n=4$.  A minimal generating set for the ideal $I_{\Z_4}$ consists of $151$ binomials:
\begin{itemize}
	\item $6$ quadrics, 
	\item $4$ cubics, 
	\item $93$ quartics and 
	\item $48$ quintics.  
\end{itemize}
Some of these violate the requirement that each dyad can be observed in only one state. 
As it is impractical to write all of these binomials down, we will list just a few of those that are statistically meaningful (i.e. respect the requirement of at most one edge per dyad at any time).  As expected, the quadrics and the cubics are simply the generators of $I_{\Z_3}$ for the four $3$-node subnetworks of the four-node network.  As we've seen, the quadrics are not of interest.  The cubics represent the three-cycles.  Here is a list of sample quartics, written in binomial form:
\begin{align*}
	p_{12}(1,1)p_{34}(1,1)p_{23}(0,0)p_{14}(0,0) & -  p_{12}(0,0)p_{34}(0,0)p_{23}(1,1)p_{14}(1,1),\\ 
	 p_{23}(1,1)p_{14}(1,1)p_{13}(0,0)p_{24}(0,0) &-   p_{23}(1,0)p_{14}(1,0)p_{13}(0,1)p_{24}(0,1),\\ 
	 p_{23}(1,1)p_{14}(1,1) p_{12}(0,0)p_{34}(0,0) &- p_{12}(1,0)p_{23}(1,0)p_{34}(1,0)p_{14}(0,1),\\ 
	 p_{12}(0,0)p_{23}(1,1)p_{34}(0,1)p_{14}(1,0) &- p_{12}(1,0)p_{23}(1,0)p_{34}(1,1)p_{14}(0,0) .\\  
\end{align*}
Finally, we list some representative quintics:
\begin{align*}
	 p_{12}(0,0)p_{23}(1,1)p_{34}(0,1)p_{14}(0,1)p_{24}(1,0)  &-  p_{12}(0,1)p_{23}(1,0)p_{34}(1,1)p_{14}(0,0)p_{24}(0,1) ,\\ 
	 p_{12}(1,0)p_{23}(1,0)p_{14}(0,0)p_{13}(1,1)p_{24}(1,0)  &-  p_{12}(0,1)p_{23}(1,1)p_{14}(1,0)p_{13}(1,0)p_{24}(0,0).\\ 
\end{align*}

This set of Markov moves is quite more complex then the  $10$ moves originally described by Holland and Leinhardt for the $4$-node case.
We will postpone any further analysis of these binomials until the next section. For now, let us note that all of them preserve the in- and out- degree distributions of the nodes in the network.  After we study the other two cases for $\rho_{ij}$, we will see a recurring underlying set of moves which can be used to understand the ideals.

\smallskip
\subsubsection{Case II: constant reciprocation} {\bf{($\rho_{ij}=\rho$)}}
\label{subsec:smallMarkov-constantRho}

Now we introduce one more row to the zero-$\rho$ design matrix $\Z_n$ to obtain the constant-$\rho$ matrix $\C_n$.
Namely, this row represents the constant $\rho$ added to those columns indexed by $p_{ij}(1,1)$ for all $i,j\in [n]$.
It is filled with the pattern $0,0,0,1$ repeated as many times as necessary.  For example, the design matrix for the $2$-node network is as follows:
\begin{align*}
    \C_2 = \mtx
    1&1&1&1\\
    0&1&0&1\\
    0&0&1&1\\
    0&0&1&1\\
    0&1&0&1\\
    0&1&1&2\\
    0&0&0&1
    \mtxend 
    \begin{matrix}
	    \lambda_{12}\\ \alpha_1 \\ \alpha_2 \\ \beta_1 \\ \beta_2 \\ \theta \\ \rho
    \end{matrix} 
\end{align*}
In this case the ideal is empty (there is nothing in the kernel of $\C_2$), which is expected since the only relation in the case of $\rho_{ij}=0$ required that there is no reciprocation. Here, the bidirected edge is valued differently then the two single edges in a dyad; this is the meaning of the last row of the design matrix.

For the $3$-node network, the Markov move consist only of the cubic from the case $\rho_{ij}=0$:
\begin{align*}
	    p_{12}(0,1) p_{13}(1,0) p_{23}(0,1) & -p_{12}(1,0) p_{13}(0,1) p_{23}(1,0).
\end{align*}
On a side note, even the Graver basis consists only of this move and $15$ other non-applicable moves (those which ignore the single-edged dyad assumption).

Let $n=4$. The software {\tt 4ti2} outputs a minimal generating set of the ideal $I_{\C_4}$ consisting of:
\begin{itemize}
	\item  $4$ cubics,
	\item $57$ binomials of degree $4$,
	\item $72$ of degree $5$,
	\item $336$ of degree $6$,
	\item $48$ of degree $7$, and 
	\item $18$ of degree $8$.
\end{itemize}
Out of this large (!) set, the applicable Markov moves are the same as in the case $\rho_{ij}=0$ with a few degree-six binomials added, such as:
\begin{align*}
	p&_{12}(0,0)p_{13}(1,1)p_{14}(1,1)p_{23}(0,1)p_{24}(1,0)p_{34}(0,0) -  \\
		&   p_{12}(1,1)p_{13}(0,1)p_{14}(1,0)p_{23}(0,0)p_{24}(0,0)p_{34}(1,1).
\end{align*}

\smallskip
\subsubsection{Case III: edge-dependent reciprocation} {\bf{($\rho_{ij}=\rho+\rho_i+\rho_j$)}}
\label{subsec:smallMarkov-RhoEdgeDependent}

To construct the design matrix $\E_n$ for this case, we start with the matrix $\C_n$ from the case $\rho_{ij}=\rho$,
and introduce $n$ more rows indexed by $\rho_1,\dots, \rho_n$.
Every fourth column of the new matrix, indexed by $p_{ij}(1,1)$, has two nonzero entries: a $1$ in the rows corresponding to $\rho_i$ and $\rho_j$.
For example, when $n=2$, the matrix looks like this:
\begin{align*}
    \E_2 = \mtx
    1&1&1&1\\
    0&1&0&1\\
    0&0&1&1\\
    0&0&1&1\\
    0&1&0&1\\
    0&1&1&2\\
    0&0&0&1\\
    0&0&0&1\\
    0&0&0&1
    \mtxend 
    \begin{matrix}
	    \lambda_{12}\\ \alpha_1 \\ \alpha_2 \\ \beta_1 \\ \beta_2 \\ \theta \\ \rho \\ \rho_1 \\ \rho_2
    \end{matrix} 
\end{align*}
This is a full-rank matrix so the ideal for the $2$-node network is empty.

With $n=3$ we get the expected result;  the ideal $I_{\E_3}$ is the principal ideal
$$ I_{\E_3} = ( p_{12}(1,0)p_{23}(1,0)p_{13}(0,1) - p_{12}(0,1)p_{23}(0,1)p_{13}(1,0) ) . $$

With $n=4$ we get the first interesting Markov moves for the edge-dependent case.
The software {\tt 4ti2} outputs a minimal generating set of the ideal $I_{\E_4}$ consisting of:
\begin{itemize}
	\item  $4$ cubics,
	\item $18$ binomials of degree $4$,
	\item $24$ of degree $5$.
\end{itemize}
The cubics, as usual, represent re-orienting a $3$-cycle.  Similarly some of the quartics represent $4$-cycles.  And then we get a few more binomials, of the following types:
\begin{align*}
	p&_{13}(0,0) p_{24}(0,0)  p_{14}(0,1)  p_{23}(0,1)   - \\
	  & p_{13}(0,1) p_{24}(0,1)  p_{14}(0,0)  p_{23}(0,0)
\end{align*}
of degree four, and 
\begin{align*} 
	p&_{13}(0,0) p_{24}(0,0) p_{14}(0,1) p_{12}(1,0) p_{23}(1,0) - \\
	   & p_{13}(1,0) p_{24}(0,1) p_{14}(0,0) p_{12}(0,1) p_{23}(0,0)
\end{align*}
of degree five.
Note that these two are just representatives; we may, for example, replace every $p_{ij}(0,0)$ in each of them by $p_{ij}(1,1)$, and get other Markov moves which are minimal generators of the toric ideal $I_{\E_4}$.

\smallskip

\subsection{From the $p_1$ model to an edge subring of a graph}
\label{section:simplified-model-and-graphs}

A careful reader will have noticed a pattern in the moves that have appeared so far.  To that end, 
let us single out two special submatrices that appear in the design matrices in each of the three cases.  
\begin{enumerate}
	\item
		First, for each case, we consider the matrix of the \emph{simplified} model obtained from the $p_1$ model by 
		simply forgetting the normalizing constants $\lambda_{ij}$.  Let us denote these simplified matrices  by 
		 $\tilde\Z_n$, $\tilde\C_n$ and $\tilde\E_n$. Note that ignoring $\lambda_{ij}$'s results in zero columns 
		 for each column indexed by $p_{ij}(0,0)$, 
		 and so we are effectively also ignoring all $p_{ij}(0,0)$'s.
		Hence, the matrices $\tilde\Z_n$, $\tilde\C_n$ and $\tilde\E_n$ have ${{n}\choose{2}}$ less rows and ${{n}\choose{2}}$
		less columns than $\Z_n$, $\C_n$ and $\E_n$, respectively.
	\item
		The second special matrix will be denoted by $\A_n$ and is common to all three cases. It is obtained from 
		$\tilde\Z_n$, $\tilde\C_n$ or $\tilde\E_n$
		by ignoring the columns indexed by $p_{ij}(1,1)$ for all $i$ and $j$, and then removing any zero rows.  
\end{enumerate}
While these constructions may seem artificial at a first glance, we will soon see that they are helpful in effectively describing 
Markov bases for $p_1$ model for any $n$.

\smallskip

Let us consider an example.
The ideal of $I_{\A_4}$ is generated by the four cubics representing the $3$-cycles, and six quadrics, each of which represents the following move for some choice of $i,j,k,l\in\{1,\dots,4\}$:
\begin{verbatim}
	Replace the edges from i to j and from l to k 
		by the edges from i to k and from l to j.
\end{verbatim}
Graphically, 
the move simply exchanges the heads of the directed arrows while keeping the tails fixed.

It turns out that these moves are basic building blocks for the Markov moves for the $4$-node network, decorated with some other edges to ensure homogeneity, and sometimes this decoration can be quite non-obvious.   They depend on the homogeneity requirements which are there for all three cases of $p_1$, but also on the way that bi-directed edges might interact with ``regular'' edges, specially in the case of no reciprocation.  
In particular, 
we will see (Theorems \ref{thm:simplified-noRho-decomposition}, \ref{thm:simplified-EdgeRho-decomposition}) that the ideals of the simplified models are a sum of the ideal $I_{\A_n}$ and another nice toric ideal. 
 It will then follow that the ideal of our model is a multi-homogeneous piece of the ideal of the simplified model (Theorem \ref{thm:geometry-of-p1:multihomog}).  
 Equivalently, the corresponding varieties are obtained by slicing the simplified-model varieties with hyperplanes.
The upshot of the decomposition is that the Markov moves can be obtained by overlapping simple moves.
 
\begin{eg}\label{eg:overlap}
	The following binomial is a generator of the ideal $I_{\Z_4}$:
	\begin{align*}
		p_{12}(1,0) p_{13}(1,1) p_{23}(1,0) p_{24}(1,0)  -  p_{12}(0,1) p_{13}(1,0) p_{14}(1,0) p_{23}(1,1). 
	\end{align*}
	The move itself is equivalent to performing a sequence of two simple moves, namely:
	\begin{verbatim}
		replace the cycle 1 -> 2 -> 3
			by the cycle 1 <- 2 <- 3 ,
	\end{verbatim}
	followed by
	\begin{verbatim}
		Replace the edges from 1 to 3 and from 2 to 4 
			by the edges from 1 to 4 and from 2 to 3.
	\end{verbatim}
	This ''decomposition" depends on the fact that reciprocation is zero, so that the double edge is valued the same as two regular edges.
\end{eg}
The following example illustrates that not all Markov moves are obtained in the same fashion. 
\begin{eg}\label{eg:overlap-and-lift}
	Consider the case of edge-dependent reciprocation on $n=4$ nodes. The following degree-five binomial appears as a minimal generator of the ideal $I_{\E_4}$,
	for a choice of $1\leq i,j,k,l \leq n$:
	\begin{align*}
		p_{ij}(1,0) p_{ik}(0,0) p_{il}(0,1) p_{jk}(1,0) p_{jl}(0,0) - p_{ij}(0,1) p_{ik}(1,0) p_{il}(0,0) p_{jk}(0,0) p_{jl}(0,1).
	\end{align*}
	Clearly this move can be obtained by the following sequence of simple moves:
	\begin{verbatim}
		Replace the edges from l to i and from j to k 
			by the edges from l to k and from j to i,
	\end{verbatim}
	followed by
	\begin{verbatim}
		Replace the edges from i to j and from l to k 
			by the edges from i to k and from l to j.
	\end{verbatim}
	Note that we do not remove or add the empty edge represented by $p_{ij}(0,0)$; but these variables are required by homogeneity.
\end{eg}
%


\subsection{The toric ideal of the common submatrix}{\bf ($\A_n$)}
\label{subsec:common-submatrix-ideal}

Focusing on the submatrix $\A_n$ reveals additional structure which can be studied using some standard algebraic techniques. 
To that end, we recall the following standard definition (see \cite{Vill}).
\begin{defn}
	Let $k$ be any field (e.g. $k=\mathbb C$), $G$ be any graph and $E(G)$ the set of its edges.  
	If we consider the vertices of $G$ to be variables, then the edges of $G$ correspond to degree-two monomials in those 
	variables. The ring 
	$$k[G]:=k[ xy \phantom{i}  | \phantom{i}  (x,y) \in E(G)]$$
	is called the \emph{monomial subring} or the \emph{edge subring} of the graph $G$.  \\
	Its ideal of relations is called the \emph{toric ideal of the edge subring}.
\end{defn}
We will be interested in special graphs.  
Let $G_n:=K_{n,n}\backslash\{(i,i)\}_{i=1}^n$ be the complete bipartite (undirected) graph $K_{n,n}$ on two sets of with $n$ vertices each, 
but with the vertical edges $(i,i)$ removed.  
If we label one set of vertices $\alpha_1,\dots,\alpha_n$ and the other set $\beta_1,\dots,\beta_n$, then our graph $G_n$ has edges
$(\alpha_i,\beta_j)$ for all $i\neq j$.  Thus we are interested in the ring
$$k[G_n]:=k[\alpha_i\beta_j \phantom{i} | \phantom{i} (\alpha_i,\beta_j)\mbox{ is an edge of } G_n].$$

\begin{lm} \label{lm:homog-part-of-our-ideal-corresponds-to-toric-ideal-of-graph}
    With notation just introduced, the ideal $I_{\A_n}$ is the toric ideal of the monomial subring $k[G_n]$.
\end{lm}
\begin{proof}
    Returning to our definition of $\A_n$, we see that it is the incidence matrix of the graph $G_n$; 
    that is, the columns of the matrix correspond to the exponents of the monomials representing the edges of $G_n$.
    Thus the claim readily follows by the definition of the toric ideal of $k[G_n]$ from Section 8.1. in \cite{Vill}.
\end{proof}
We will use this correspondance to obtain a Gr\"obner basis of $I_{\Z_n}$ and $I_{\E_n}$. But first, 
 let us introduce one more concept which is crucial to the description of a Gr\"obner basis of $I_{\A_n}$.
\begin{defn}\rm
    \label{defn:binomials-from-cycles}
    Following \cite{OH} and \cite{Vill}, we say that a binomial $f_c$
    \emph{arises from} a cycle $c$ of $G_n$ if:
    \begin{enumerate}
    \item $c$ is a cycle of the graph $G_n$; namely, $c$ is a closed walk on a subset of the vertices of $G_n$ with the property that
          no vertex is visited twice:
          $$  c = ( \alpha_{i_1}, \beta_{j_1}, \alpha_{i_2}, \beta_{j_2}, \dots, \beta_{j_k}, \alpha_{i_1} ) ,
          $$
    \item $f_c$ is the binomial obtained by placing the variables corresponding to the edges of the cycle $c$ alternately in one and then the other
          monomial of $f_c$:
          $$ f_c := f_c^+ - f_c^-
          $$
          with
          $$ f_c^+ := p_{i_1,j_1}(1,0) p_{i_2,j_2}(1,0) \dots p_{i_{k-1},j_{k-1}}(1,0) p_{i_i,j_k}(1,0)
          $$
          and
          $$ f_c^- := p_{i_2,j_1}(1,0) p_{i_3,j_2}(1,0) \dots p_{i_k,j_{k-1}}(1,0) p_{i_1,j_k}(1,0)   ,
          $$
          where for ease of notation we have let $p_{ij}(1,0):=p_{ji}(0,1)$ if $i>j$.
    \end{enumerate}
\end{defn}
There are finitely many cycles in $G_n$, though they may be nontrivial to enumerate. However, we will use this description to provide (theoretically)
an explicit Gr\"obner basis for our toric ideal.  In practice, one can use the program {\tt 4ti2} to obtain the binomials fairly quickly. 

Gr\"obner  bases, and even Graver bases, for the ideals $I_{\A_n}$ are known (\cite{OH}, \cite{Vill}):
\begin{theorem}[{{Bases of the ideal of the common submatrix}}]
\label{prop:G-basis-of-common-submatrix}

	Let $\G_n$ be the set of binomials arising from the cycles of the graph $G_n$. Then $\G_n$ is a Gr\"obner basis of $I_{\A_n}$.
	Thus it is also a Markov basis of $I_{\A_n}$, but not necessarily a minimal one. 
	
	Moreover, this set is a Graver basis of the ideal, and it coincides with 
	the universal Gr\"obner basis.
\end{theorem}
\begin{proof}[Proof (outline)]
    We will use Lemma
    \ref{lm:homog-part-of-our-ideal-corresponds-to-toric-ideal-of-graph}
    and the appropriate results from \cite{OH} and \cite{Vill} about the toric ideals of bipartite graphs.
    Note that there is quite a bit of vocabulary used in the cited results, but we have only defined those terms which we
    use in our description of the Gr\"obner basis in the Theorem.\\
    Oshugi and Hibi in \cite{OH} (Lemma 3.1.), and also Villarreal in \cite{Vill} (Proposition 8.1.2)
    give a generating set for $I_{{\A}_n}$.
    Moreover, Lemma 1.1. of \cite{OH} implies that these generators are in fact the Graver basis for $I_{{\A}_n}$. 
    On the other hand, from Chapter 8 of \cite{Vill} we see that precisely these generators in fact correspond to circuits,
    as well as the universal Gr\"obner basis for $I_{{\A}_n}$.  
    It follows that the circuits equal the Graver basis for our ideal $I_{{\A}_n}$.
    Avoiding technical details, we will just state that \emph{circuits} are a special subset
    of the Graver basis: they are minimal with respect to inclusion.
    The binomials in the Graver basis are given by the \emph{even cycles} of the graph $G_n$ (Corollary 8.1.5. in \cite{Vill}). But since our graph
    is bipartite, all of the cycles are even (Proposition 6.1.1. in \cite{Vill}), so it suffices to say that the Graver basis of $I_{{\A}_n}$
    consists of \emph{binomials arising from the cycles of the graph $G_n$}.
    Since the universal Gr\"obner basis (which equals Graver basis in this case) is by definition a Gr\"obner basis with respect to any
    term order, the proof is complete.  The binomials arising from the cycles of the graph $G_n$ form the desired set $\G_n$.
\end{proof}

In addition, we have a nice recursive way of constructing the binomials in $\G_n$.
\begin{prop}
		\label{prop:recursion-for-Gbasis-of-A(noRho)}
    The set $\G_n$ consists of precisely the binomials in $\G_{n-1}$ for all of the $n-1$-node subnetworks,
    together with the binomials arising from the cycles of the graph $G_n$ which pass through either $\alpha_i$ or $\beta_i$ for
    \emph{each} $i$ between $1$ and $n$.  
     %
\end{prop}
\begin{proof}
    The last condition can be restated as follows:
    we are looking for the primitive binomials $f$
    such that for each node $i$ in the random network, there exists an edge $(i,j)$ such that the variable $p_{ij}(a,b)$ appears in one of the monomials
    of $f$.\\
    The reduction to the $n-1$-node subnetworks is
    immediate from Proposition 4.13. in \cite{St}.  Namely, the design matrices for the $n-1$-node networks are submatrices of ${\A}_n$, hence the
    Graver basis of $I_{{\A}_{n-1}}$ equals that of $I_{{\A}_n}$ involving at most $n-1$ nodes.
\end{proof}
An example of a degree $5$ binomial in $I_{\A_5}$  which uses all $5$ nodes of the network is
$$ p_{14}(1,0)p_{15}(0,1)p_{23}(1,0)p_{24}(0,1)p_{35}(1,0) - p_{14}(0,1)p_{15}(1,0)p_{23}(0,1)p_{24}(1,0)p_{35}(0,1) .$$
The first term represents a cycle $1\to 4\to 2\to 3\to 5\to 1$, while the second term represents the same cycle with opposite orientation.
In terms of the graph $G_5$, each of the monomials represent edges, in alternating order, of the cycle
$$(\alpha_1,\beta_4, \alpha_2, \beta_3, \alpha_5, \beta_1, \alpha_4, \beta_2, \alpha_3, \beta_5, \alpha_1).$$

\begin{rmk}\rm
	In the Proposition above, all binomials have squarefree terms. This means that the initial ideal of the toric ideal 
	is squarefree. Using a theorem of Hochster (\cite{Hochster}), this implies that the coordinate rings of the corresponding varieties 
	are arithmetically Cohen-Macaulay.  This is a powerful algebraic-geometric property that is not encountered too often.
\end{rmk}
\begin{rmk}\rm
	There is a special property of the design matrix- \emph{unimodularity}- which, if satisfied, implies that all initial ideals or its toric ideal are squarefree, and also 
	that circuits equal the Graver basis.  In general, for the design matrix of the $p_1$ model it does \emph{not} hold.
\end{rmk}

We are now ready to embark on a more detailed study of the simplified models.

\smallskip

\subsubsection{Case I: no reciprocation} {\bf ($\rho_{ij}=0$, simplified model)}
\label{subsec:simplified-noRho}

Let's start with the simplest nontrivial example: $n=2$.
The design matrix is:
\begin{align*}
    \tilde\Z_2 = \mtx
    1&0&1\\
    0&1&1\\
    0&1&1\\
    1&0&1\\
    1&1&2
    \mtxend .
\end{align*}
The rows of $\tilde\Z_2$ are indexed by $\alpha_1$, $\alpha_2$, $\beta_1$, $\beta_2$, and $\theta$,
 while the columns are indexed by $p_{12}(1,0)$, $p_{12}(0,1)$, $p_{12}(1,1)$.
One easily checks that the ideal $I_{\tilde\Z_2}$ is the principal ideal
$$I_{\tilde\Z_2} = ( p_{12}(1,0) p_{12}(0,1) - p_{12}(1,1) )$$
and thus this single binomial is a Markov basis, and also a Gr\"obner basis with respect to any term order.

Next, let $n=3$.
The toric ideal $I_{\tilde\Z_3}$ is minimally generated by the following set of binomials:
\begin{align*}
    p_{23}(0,1) p_{23}(1,0) &-p_{23}(1,1),\\
    p_{13}(0,1) p_{13}(1,0) &-p_{13}(1,1),\\
    p_{12}(0,1) p_{12}(1,0) &-p_{12}(1,1),\\
    p_{12}(0,1) p_{13}(1,0) p_{23}(0,1) & -p_{12}(1,0) p_{13}(0,1) p_{23}(1,0).
\end{align*}
It is interesting to note that the first $3$ generators are the ''trivial'' ones (as seen in case $n=2$).

For the network on $n=4$ nodes, we get the first interesting Markov basis elements. Namely, there are inhomogeneous binomials of degree $2$ which
represent the ''obvious'' relations, which are of the form:
$$p_{ij}(0,1) p_{ij}(1,0) -p_{ij}(1,1).$$
Next, there are squarefree quadrics (homogeneous binomials of degree $2$):
$$ p_{ik}(1,0) p_{jl}(1,0) -p_{il}(1,0) p_{jk}(1,0),$$
and there are degree-three binomials that resemble the degree-three generator of the ideal $I_{\tilde\Z_3}$:
$$p_{ij}(0,1) p_{ik}(1,0) p_{jk}(0,1) -p_{ij}(1,0) p_{ik}(0,1) p_{jk}(1,0) .$$
Note that if we write $p_{jk}(1,0)$ and $j>k$, then we mean $p_{kj}(0,1)$, since for all $p_{jk}(a,b)$ we assume $j<k$.

In order to obtain a Gr\"obner basis for the ideal $I_{\tilde\Z_n}$, we can use what we know about $I_{\A_n}$. 
The advantage of studying $I_{\A_n}$ over $I_{\tilde\Z_n}$ is that it is a \emph{homogeneous} ideal; 
that is, both monomials in each binomial appear with the same degree.  
One can see that the ideal is homogeneous by inspecting the columns of the matrix: each column has the same $1$-norm.
The ideal of the simplified model admits a nice decomposition:
\begin{theorem} [{Decomposition of the ideal of the simplified model $\tilde\Z_n$}]
\label{thm:simplified-noRho-decomposition}
    With the notation as above, $I_{\tilde\Z_n} = I_{{\A}_n} + T $, where $T$ is the ideal generated by the binomials of the form
    $$p_{ij}(0,1) p_{ij}(1,0) -p_{ij}(1,1)$$
    for all pairs of nodes $i<j$.
\end{theorem}
\begin{proof}
    One inclusion ($\supset$) is clear.  To prove the other inclusion, consider $f\in I_{\tilde\Z_n}$ such that $p_{ij}(1,1)$ divides one of its terms for some $i,j$.
    Since the ideal is toric, it suffices to consider the case when $f$ is a binomial.  Then it can be written as
    $f= p_{ij}(1,1) m_1 - m_2$ where $m_1$ and $m_2$ are monomials.  But then we can write 
    \begin{align*}
        f = (p_{ij}(1,0)p_{ij}(0,1) m_1 - m_2) - m_1 (p_{ij}(0,1) p_{ij}(1,0) -p_{ij}(1,1)).
    \end{align*}
    Repeating this procedure if necessary, we can write $f$ as a combination of binomials whose terms are not divisible by any of the variables
    $p_{ij}(1,1)$ and those binomials that generate $T$.  To conclude, note that those not divisible by any of the $p_{ij}(1,1)$ are in the ideal $I_{{\A}_n}$ by definition.
\end{proof}

Combining the above results, we obtain a Markov (Gr\"obner) basis for the ideal of the simplified model: 
\begin{theorem}[{Gr\"obner basis for the simplified model $\tilde\Z_n$}]
 \label{thm:Gbasis-of-A(noRho)-using-graphs}
    Let $\G_T$ be the binomial generators of $T$,  that is,
    $$\G_T:=\{p_{ij}(0,1) p_{ij}(1,0) -p_{ij}(1,1)\}
    $$ for all pairs of nodes $i<j$.
    Let $\G_n$ be the set of binomials arising from the cycles of $G_n$, as in Theorem (\ref{prop:G-basis-of-common-submatrix}).

    Then the ideal of the simplified model $I_{\tilde\Z_n}$ has a Gr\"obner basis, and thus a Markov basis, consisting of the union of the
     binomials in  $\G_n$ and $\G_T$.
\end{theorem}
\begin{proof}
    First we will show that the set $\G_T$ actually forms a Gr\"obner basis for $T$.
    Namely, pick an elimination order where the variables $p_{ij}(1,1)$ are greater then the remaining variables. Then, the degree-one terms
    of the binomials in $\G_T$ are the initial terms.  Since they are all relatively prime, the given generators form a Gr\"obner basis
    as all S-pairs reduce to zero (for details about Gr\"obner basis computations, the reader should refer to Buchberger's criterion \cite{Buch}
    as well as standard references \cite{CLO} and \cite{St}). \\
    Next, take any Gr\"obner basis $\G_n$ for the ideal $I_{\A_n}$ with respect to some order $\<$.
    According to Lemma \ref{thm:simplified-noRho-decomposition}, $\G_n\cup \G_T$ generates the ideal $I_{\tilde\Z_n}$.  Let $\<'$ be the refinement of the term order $\<$
    by the elimination order used for $\G_T$.  The we get that $\G_n\cup \G_T$ is in fact  a Gr\"obner basis for $I_{\tilde\Z_n}$ with respect to $\<'$,
    since no initial term in $\G_T$ appears as a term in $\G_n$ and thus all the S-pairs will again reduce to zero.

    What remains to be done is to find a Gr\"obner basis $\G_n$ for the ideal $I_{\A_n}$.  
    But this was done in Theorem (\ref{prop:G-basis-of-common-submatrix}). 
\end{proof}



To derive \emph{minimal} Markov bases for the simplified model, it remains to find a minimal generating set for the ideal of the edge subring of an incomplete bipartite graph.
To the best of our knowledge, there isn't a result that states these are generated in degrees two and three.  
In particular, we claim:
\begin{conj}
  \label{thm:minimal-Markov-simplified-noRho}
    A {minimal} Markov basis for the ideal $I_{\tilde\Z_n}$ consists of the elements of degrees $2$ and $3$. All inhomogeneous elements are of the form
    $$p_{ij}(0,1) p_{ij}(1,0) -p_{ij}(1,1).$$
    All quadrics are of the form $$ p_{ik}(1,0) p_{jl}(1,0) -p_{il}(1,0) p_{jk}(1,0),$$ where $i$, $j$, $k$, $l$ vary over all possible four-node
    subnetworks. \\
    All cubics are of the form
    $$p_{ij}(0,1) p_{ik}(1,0) p_{jk}(0,1) -p_{ij}(1,0) p_{ik}(0,1) p_{jk}(1,0), $$
    where $i$, $j$, and $k$ vary over all possible triangles ($3$-cycles) in the network.
\end{conj}


\subsubsection{Case II: constant reciprocation} {\bf{($\rho_{ij}=\rho$, simplified model)}}
\label{subsec:simplified-constantRho}

There is a small but crucial difference between $\tilde\Z_n$ and $\tilde\C_n$: one more row, 
 representing the constant $\rho$ added to those columns indexed $p_{ij}(1,1)$ for all $i,j\in [n]$.
This row is filled with the pattern $0,0,1$ repeated as many times as necessary.
Note that adding the extra row makes this ideal {\emph{homogeneous}}, that is, the two terms of each binomial have the same degree.
(Note that homogeneity is easy to verify: the dot product of each column of $\tilde\C_n$ with the vector $w=[1,\dots,1,-n-1]$ results in the
vector $[1,\dots,1]$, as required, for example, by Lemma 4.14. in \cite{St}.)

For two nodes, the ideal is trivial.

For $n=3$ nodes, the Markov basis consists of $4$ binomials of degree $3$ of the following forms:
\begin{align*}
 p_{ik}(1,0) p_{ik}(0,1) p_{jk}(1,1) &- p_{ik}(1,1) p_{jk}(1,0) p_{jk}(0,1),\\
 p_{ij}(1,0) p_{ik}(0,1) p_{jk}(1,0) &- p_{ij}(0,1) p_{ik}(1,0) p_{jk}(0,1),
 \end{align*}
 and $6$ of degree $4$ of the form:
$$ p_{ij}(0,1)^2 p_{ik}(1,1) p_{jk}(0,1) - p_{ij}(1,1) p_{ik}(0,1)^2 p_{jk}(1,0). $$
For $n\geq 4$, the Markov bases consist of elements of the above type, but also include quadrics of the form:
\begin{align*}
p_{ik}(1,0) p_{jl}(1,0) &- p_{il}(1,0) p_{jk}(1,0),\\
p_{ij}(1,1) p_{kl}(1,1) &- p_{ik}(1,1) p_{jl}(1,1).
\end{align*}

\begin{conj}  [{Minimal Markov basis of the simplified model $\tilde\C_n$}]
	The ideal $I_{\tilde\C_n}$ is generated in degrees $2$, $3$, and $4$, and the binomials are of the form described above.
\end{conj}
Note: due to the existence of $\rho_{ij}$, we do not get the relations $\G_T$ from the case $\rho_{ij}=0$.
Recall they were of the form $p_{ij}(1,0)p_{ij}(0,1)-p_{ij}(1,1)$. They cannot be in the ideal in this case, since the bi-directed edge between $i$ and $j$
is valued differently then the two single edges $i\to j$ and $j\to i$.

Gr\"obner bases or even Markov bases in this case remain an open problem that we will continue to study.

\smallskip

\subsubsection{Case III: edge-dependent reciprocation} {\bf{($\rho_{ij}=\rho+\rho_i+\rho_j$, simplified model)}}
\label{subsec:simplified-RhoEdgeDependent}

To construct the design matrices $\tilde\Z_n$ for this case, we start with the matrix $\tilde\C_n$ from the case $\rho_{ij}=\rho$,
and introduce $n$ more rows indexed by $\rho_1,\dots, \rho_n$.
Every third column of the new matrix, indexed by $p_{ij}(1,1)$, has two nonzero entries: a $1$ in the rows corresponding to $\rho_i$ and $\rho_j$.
For example, when $n=3$, the matrix is:
\begin{align*}\tilde\E_3 =  \mtx
1&0&1&1&0&1&0&0&0\\
0&1&1&0&0&0&1&0&1\\
0&0&0&0&1&1&0&1&1\\
0&1&1&0&1&1&0&0&0\\
1&0&1&0&0&0&0&1&1\\
0&0&0&1&0&1&1&0&1\\
1&1&2&1&1&2&1&1&2\\
0&0&1&0&0&1&0&0&1\\
0&0&1&0&0&1&0&0&0\\
0&0&1&0&0&0&0&0&1\\
0&0&0&0&0&1&0&0&1
\mtxend .
\end{align*}
The kernel of this matrix is generated by one vector, and in fact the ideal $I_{\tilde\E_3}$ is the principal ideal
$$ I_{\tilde\E_3} = ( p_{12}(1,0)p_{23}(1,0)p_{13}(0,1) - p_{12}(0,1)p_{23}(0,1)p_{13}(1,0) ) . $$
When we calculate the ideals for larger networks, some more familiar binomials appear.

\begin{theorem}  [{Decomposition of the simplified model $\tilde\E_n$}]
\label{thm:simplified-EdgeRho-decomposition}
 	Let $I_{\A_n}$ be the ideal as in Section \ref{subsec:common-submatrix-ideal}.
    Let $Q$ be the ideal generated by the quadrics of the form
    \begin{align*}
        p_{ij}(1,1)p_{kl}(1,1)-p_{ik}(1,1)p_{jl}(1,1)
    \end{align*}
    for each set of indices $1\leq i,j,k,l \leq n$. Then:
    \begin{align*}
        I_{\tilde\E_n} = I_{\A_n} + Q
    \end{align*}
    for every $n\geq 4$.
\end{theorem}
\begin{proof}
    Consider the submatrix of $\tilde\E_n$ consisting of those columns that have only zeros in the last $n$ rows.  Erasing those zeros, we get
    precisely the columns of $\A_n$.  Thus $I_{\tilde\E_n} \supset I_{\A_n}$.

    Similarly, the ideal of relations among those columns that have nonzero entries in the last $n$ rows is also contained in $I_{\tilde\E_n}$.
    We will now show that this ideal equals $Q$.
    To simplify notation, let $M$ be the matrix consisting of the columns in question, that is, those that are indexed by $p_{ij}(1,1)$ for all pairs
    $i<j$. Recalling the definition of the action of the simplified parametrization map on $p_{ij}(1,1)$:
    \begin{align*}
        p_{ij}(1,1) \mapsto \alpha_i \beta_j \alpha_j \beta_i \theta \rho \rho_j \rho_j,
    \end{align*}
    and the fact that the last $n$ rows of $M$ are indexed by $\rho_1,\dots,\rho_n$,
    we see that to study the toric ideal $I_M$ it suffices to study the ideal $I_{M'}$ where $M'$ is the submatrix consisting of the
    last $n$ rows of $M$.
    But $I_{M'}$ is a well-studied toric ideal! Namely, $M'$ agrees with the incidence matrix of the complete graph $K_n$ on $n$ vertices:
    for each pair $\{i,j\}$ there exists an edge $(i,j)$.  Therefore the toric ideal $I_{M'}$ agrees with  the toric ideal of the edge ring $k[K_n]$,
    where the vertices of $K_n$ are labelled by the nodes $1,\dots,n$, and an edge between $i$ and $j$ represents the variable $p_{ij}(1,1)$.
    It is well-known (see, for example, \cite{OH}) that $I_{M'}$ is generated by quadrics.
    Moreover, from \cite{CoNa} and references given therein (see also Corollary 9.2.3. in \cite{Vill}) 
    it follows that these quadrics be interpreted as $2\times 2$-minors of a certain tableaux.
    But our definition of the ideal $Q$ is precisely the set of those $2\times 2$-minors, thus $I_{M'}=Q$ and we obtain $I_{\tilde\E_n} \supset Q$.

    To complete the proof, we need to show the reverse inclusion.  To that end, let $f=f^+ -f^- \in I_{\tilde\E_n}$.
    If no $p_{ij}(1,1)$ divides either term of $f$, then $f \in I_{\A_n}$ and we are done.
    On the other hand, suppose $p_{ij}(1,1) | f^+$ for some pair $i,j$.
    Then the definition of $I_{\tilde\E_n}$ implies that one of the following conditions are satisfied:
    \begin{enumerate}
        \item $p_{ij}(1,1) | f^- $. But then the binomial $f$ is not primitive, and thus it cannot be required in any minimal generating set, or a
            Gr\"obner basis, of $I_{\tilde\E_n}$.
        \item $p_{kl}(1,1) | f^+ $ for some other pair of indices $k,l$. But this in turn implies that
        \begin{enumerate}
            \item $p_{ij}(1,1) p_{kl}(1,1) | f^-$ and $f$ fails to be primitive trivially; or
            \item without loss of generality, $p_{ik}(1,1) p_{jl}(1,1) | f^-$ and $f$ fails to be primitive by the quadric in $Q$; or
            \item $f^+$ is divisible by another variable $p_{st}(1,1)$, and the pattern continues.
        \end{enumerate}
    \end{enumerate}
    In general, it is clear that any $f$ whose terms are divisible by the variables representing columns of $M$ will fail to be primitive
    by one of the binomials in the ideal $Q$.  Therefore, the ideal $I_{\tilde\E_n}$ is generated by the binomials of $I_{{\A}_n}$ and $Q$,
    as desired.
\end{proof}

Next we obtain a Gr\"obner basis for our toric ideal.
Recall from Section \ref{subsec:common-submatrix-ideal} that $G_n = K_{n,n}\backslash \{(i,i)\}_{i=1}^n$.
We have seen that the graph $G_n$ played an essential role in the case when $\rho_{ij}=0$, and it will play an essential role here as well, 
since it is essential in describing the Gr\"obner bases of the common submatrix $\A_n$.  However, in order to study the Gr\"obner basis
of the ideal $Q$, we need to use graphs with more edges. 
It comes as no surprise, then, that these ideals will have a more complicated Gr\"obner basis then those from Section \ref{subsec:common-submatrix-ideal}.
Thus we need to generalize Definition \ref{defn:binomials-from-cycles} for the \emph{complete graph} $K_n$ on $n$ vertices:
\begin{defn}\rm
    \label{defn:binomials-from-walks}
    As before, following \cite{OH} and \cite{Vill}, we say that a binomial $f_w$
    \emph{arises from} an even closed walk $w$ of $K_n$ if:
    \begin{enumerate}
    \item $w$ is an even closed walk $w$ of $K_n$; namely, $w$ is a closed walk on a subset of the vertices of $K_n$ with the property that
          it is of even length (even number of edges)
          $$  w = ( \{v_1,v_2\},\dots \{v_{2k-1},v_{2k}\} ) ,
          $$
          with $1\leq v_1,\dots,v_{2k}\leq n$.  The \emph{closed} condition requires that $v_{2k}=v_1$.
    \item $f_w$ is the binomial obtained by placing the variables corresponding to the edges of the walk $w$ alternately in one and then the other
          monomial of $f_w$:
          $$ f_w := f_w^+ - f_w^-
          $$
          with
          $$ f_w^+ := p_{v_1,v_2}(1,1) p_{v_3,v_4}(1,1) \dots  p_{v_{2k-1},v_{1}}(1,1)
          $$
          and
          $$ f_w^- := p_{v_{2},v_{3}}(1,1)p_{v_{4},v_{5}}(1,1) \dots p_{v_{2k-2},v_{2k-1}}(1,1),
          $$
          where for compactness of notation we have let $p_{ij}(1,1):=p_{ji}(1,1)$ if $i>j$.
    \end{enumerate}
    We say that the even closed walk $w$ is \emph{primitive} if the corresponding binomial is primitive.
\end{defn}
Let us state the characterization of such walks from \cite{OH}:
\begin{lm}[\cite{OH}, Lemma 3.2.]
    A primitive even closed walk on a graph $G$ is one of the following:
    \begin{itemize}
        \item an even cycle of $G$;
        \item two odd cycles having exactly one common vertex;
        \item two odd cycles having no common vertex together with two more walks, both of which connect a vertex of the first cycle with
                a vertex of the second cycle.
    \end{itemize}
\end{lm}
In Section \ref{subsec:common-submatrix-ideal}, we have seen only even cycles. This is because the graph in question, namely $G_n$, was bipartite, and therefore had no odd cycles.
We are now ready to state the main result of this section.
\begin{theorem}  [{Gr\"obner basis of the simplified model $\tilde\E_n$}]
\label{thm:Gbasis-of-A(edgeRho)-using-graphs}
    The ideal $I_{\tilde\E_n}$ admits a Gr\"obner basis consisting of the binomials arising from the cycles of the bipartite graph $G_n$
    together with the binomials arising from the primitive closed even walks of $K_n$, the complete graph on $n$ vertices.
\end{theorem}
\begin{proof}
    Recall that $\G_n$ are precisely the binomials arising from the cycles of $G_n$.
    We have seen in Theorem \ref{thm:Gbasis-of-A(noRho)-using-graphs} that $\G_n$ form a Gr\"obner basis for $I_{\A_n}$ (in fact, they are
    a Graver basis).
    Since the ideals $I_{\A_n}$ and $Q$ are in disjoint sets of variables, it remains to find a Gr\"obner basis for $Q$.

    We do this by generalizing the argument in the proof of Theorem \ref{thm:Gbasis-of-A(noRho)-using-graphs}.
    Namely, from the proof of Theorem \ref{thm:simplified-EdgeRho-decomposition} we know that $Q$ is the toric ideal of the edge ring
    $k[K_n]$.  This allows us to use \cite{OH} and \cite{Vill} again, and we obtain (e.g. Lemmas 3.1. and 3.2. in \cite{OH}) that
    the Graver basis of $Q$ consists of the binomials \emph{arising from the primitive even closed walks on $K_n$}. 
    In addition, note that Theorem 9.2.1 of \cite{Vill} provides a quadratic Gr\"obner basis as well.
\end{proof}

A proof of Conjecture \ref{thm:minimal-Markov-simplified-noRho} would imply the following: 
\begin{conj} [Minimal Markov basis of the simplified model $\tilde\E_n$]
	For $n\geq 4$, the ideal $I_{\tilde\E_n}$ is minimally generated by homogeneous binomials of degrees $2$ and $3$.
	
	More precisely, the degree $2$ and $3$ binomials in $\G_n$, together with the quadratic generators of $Q$, 
	form a Markov basis for the model.
\end{conj}

\smallskip

\subsection{From the edge subring of a graph back to the  $p_1$  model}
\label{section:back-to-lambdas}

We are now ready to introduce the $\lambda_{ij}$ back into the parametrization and consider the original design matrices $\Z_n$, $\C_n$ and $\E_n$.  
Recall that they can be obtained from $\tilde\Z_n$, $\tilde\C_n$ and $\tilde\E_n$, respectively, by adding 
 ${n\choose 2}$  columns representing the  variables $p_{i,j}(0,0)$ for all $i<j$
 and ${n\choose 2}$ rows indexed by $\lambda_{ij}$.
\begin{eg}\rm
    Let $n=4$. Then, for the case in which $\rho_{i,j} = 0$, the $15 \times 24$ design matrix $\Z_4$ is:
\[
\left[
\begin{array}{cccccccccccccccccccccccc}
1 & 1 & 1 & 1 & 0 & 0 & 0 & 0 & 0 & 0 & 0 & 0 & 0 & 0 & 0 & 0 & 0 & 0 & 0 & 0 & 0 & 0 & 0 & 0\\
0 & 0 & 0 & 0 & 1 & 1 & 1 & 1 & 0 & 0 & 0 & 0 & 0 & 0 & 0 & 0 & 0 & 0 & 0 & 0 & 0 & 0 & 0 & 0\\
0 & 0 & 0 & 0 & 0 & 0 & 0 & 0 & 1 & 1 & 1 & 1 & 0 & 0 & 0 & 0 & 0 & 0 & 0 & 0 & 0 & 0 & 0 & 0\\
0 & 0 & 0 & 0 & 0 & 0 & 0 & 0 & 0 & 0 & 0 & 0 & 1 & 1 & 1 & 1 & 0 & 0 & 0 & 0 & 0 & 0 & 0 & 0\\
0 & 0 & 0 & 0 & 0 & 0 & 0 & 0 & 0 & 0 & 0 & 0 & 0 & 0 & 0 & 0 & 1 & 1 & 1 & 1 & 0 & 0 & 0 & 0\\
0 & 0 & 0 & 0 & 0 & 0 & 0 & 0 & 0 & 0 & 0 & 0 & 0 & 0 & 0 & 0 & 0 & 0 & 0 & 0 & 1 & 1 & 1 & 1\\
0 & 1 & 1 & 2 & 0 & 1 & 1 & 2 & 0 & 1 & 1 & 2 & 0 & 1 & 1 & 2 & 0 & 1 & 1 & 2 & 0 & 1 & 1 & 2\\
0 & 1 & 0 & 1 & 0 & 1 & 0 & 1 & 0 & 1 & 0 & 1 & 0 & 0 & 0 & 0 & 0 & 0 & 0 & 0 & 0 & 0 & 0 & 0\\
0 & 0 & 1 & 1 & 0 & 0 & 0 & 0 & 0 & 0 & 0 & 0 & 0 & 1 & 0 & 1 & 0 & 1 & 0 & 1 & 0 & 0 & 0 & 0\\
0 & 0 & 0 & 0 & 0 & 0 & 1 & 1 & 0 & 0 & 0 & 0 & 0 & 0 & 1 & 1 & 0 & 0 & 0 & 0 & 0 & 1 & 0 & 1\\
0 & 0 & 0 & 0 & 0 & 0 & 0 & 0 & 0 & 0 & 1 & 1 & 0 & 0 & 0 & 0 & 0 & 0 & 1 & 1 & 0 & 0 & 1 & 1\\
0 & 0 & 1 & 1 & 0 & 0 & 1 & 1 & 0 & 0 & 1 & 1 & 0 & 0 & 0 & 0 & 0 & 0 & 0 & 0 & 0 & 0 & 0 & 0\\
0 & 1 & 0 & 1 & 0 & 0 & 0 & 0 & 0 & 0 & 0 & 0 & 0 & 0 & 1 & 1 & 0 & 0 & 1 & 1 & 0 & 0 & 0 & 0\\
0 & 0 & 0 & 0 & 0 & 1 & 0 & 1 & 0 & 0 & 0 & 0 & 0 & 1 & 0 & 1 & 0 & 0 & 0 & 0 & 0 & 0 & 1 & 1\\
0 & 0 & 0 & 0 & 0 & 0 & 0 & 0 & 0 & 1 & 0 & 1 & 0 & 0 & 0 & 0 & 0 & 1 & 0 & 1 & 0 & 1 & 0 & 1\\
\end{array}
\right]
\]
\end{eg}

The following analysis applies to both the case of constant $\rho_{ij}$ and the case of edge-dependent $\rho_{ij}$,
thus we will treat them simultaneously.

There is a very concise way to describe the new toric ideal in terms of the cases when $\lambda_{ij}$ were ignored.
For a binomial $f$, we will say that $f$ is \emph{multi-homogeneous with respect to each pair $\{i,j\}$} if
the degrees in the variables indexed by the pair $\{i,j\}$ agree for the two monomials of $f$. More precisely,
if $f=f^+ - f^-$, then we require that
\begin{align*} \deg_{p_{i,j}(0,0)}(f^+) &+ \deg_{p_{i,j}(1,0)}(f^+) +\deg_{p_{i,j}(0,1)}(f^+) +\deg_{p_{i,j}(1,1)}(f^+) \\
  =
   \deg_{p_{i,j}(0,0)}(f^-) &+\deg_{p_{i,j}(1,0)}(f^-)+\deg_{p_{i,j}(0,1)}(f^-)+\deg_{p_{i,j}(1,1)}(f^-)
\end{align*}
for each pair $\{i,j\}$.
This allows us to make a simple observation.
\begin{prop}  [Geometry of the $p_1$ model]
	\label{thm:geometry-of-p1:multihomog}
    The toric ideals $I_{\Z_n}$, $I_{\C_n}$ and $I_{\E_n}$  of the $p_1$ model on an $n$-node network 
    are precisely those parts of the ideals $I_{\tilde\Z_n}$, $I_{\tilde\C_n}$ and $I_{\tilde\E_n}$, respectively, 
    which are \emph{multi-homogeneous} with respect to \emph{each} pair $\{i,j\}$. \\
	Therefore, the toric variety for the $p_1$ models is obtained by slicing the  corresponding varieties 
	of the simplified model by ${n\choose 2}$ hyperplanes.
\end{prop}
\begin{proof}
    The rows indexed by $\lambda_{ij}$ that are added to the matrices $I_{\tilde\Z_n}$, $I_{\tilde\C_n}$ and $I_{\tilde\E_n}$ 
    require precisely that the binomials in the ideal are multi-homogeneous according to the criterion given above.
    For example, consider the first row indexed by $\lambda_{1,2}$. For any binomial $f=p^{u^+}-p^{u^-}$ in the ideal, 
    the exponent vector $u^+ -u^-$ being in the kernel of the matrix means that the number of variables $p_{1,2}(\bullet,\bullet)$
    which appear in $p^{u^+}$ equals the number of variables $p_{1,2}(\bullet,\bullet)$ that appear in $p^{u^-}$.  

	For the hyperplane section statement, note that the varieties defined by 
	$I_{\tilde\Z_n}$, $I_{\tilde\C_n}$ and $I_{\tilde\E_n}$
	 live in a smaller-dimensional space than those defined by
	 $I_{\Z_n}$, $I_{\C_n}$, and $I_{\E_n}$, respectively; 
	 but we may embed them in the natural way into a higher-dimensional space.  The hyperplanes are defined by 
	 the rows indexed by the $\lambda_{ij}$'s.
\end{proof}

In general, multi-homogeneous part of any given ideal is \emph{not} generated by the multi-homogeneous generators of the 
original ideal.  
But for the ideal of the $p_1$ model we are able to use homogeneity to decompose the Markov moves into ''essential'' parts.
The upshot of this result is that analysis and computations become easier. In addition, all moves obtained this way are applicable and irredundant. 

\begin{theorem} [Essential Markov moves for the $p_1$ model] 
	\label{thm:essentialMarkovForModel}
	The Markov moves for the $p_1$ model in the case of zero and edge-dependent reciprocation 
	can be obtained from the Graver basis of the common submatrix $\A_n$, together with 
	the Markov basis for the ideal $Q$ as defined in Theorem \ref{thm:simplified-EdgeRho-decomposition}.
		
\end{theorem}
\begin{rmk}\rm
	The case of constant reciprocation is more challenging as we do not yet have a simple decomposition as in the other two cases. 
	However, a similar argument can be used to claim that the Markov moves in case of constant reciprocation can be obtained by 
	repeating the moves for the simplified model, while respecting the multi-homogeneity requirement.
\end{rmk}

Before embarking on a proof of \ref{thm:essentialMarkovForModel}, let us make the claim more precise. 
Take 
\[q=q^+-q^-\in I_{\tilde\Z_n} \mbox{ or } I_{\tilde\E_n}
\]
in the ideal of the simplified model.
The monomials $q^+$ and $q^-$ are represented by directed edges in our $n$-node network.
If $q^+$ and $q^-$ represent the same cycle with different orientation, then $q$ is multi-homogeneous already.
On the other hand, suppose
\[ q = \prod_{k=1}^{d} p_{i_k,j_k}(a_k,b_k) - \prod_{k=1}^{d} p_{s_k,t_k}(c_k,d_k)
\]
where $ p_{i_k,j_k}(a_k,b_k),p_{s_k,t_k}(c_k,d_k)\in \{p_{i,j}(1,0),p_{i,j}(0,1),p_{i,j}(1,1)\}$, $i<j$, for $1\leq k\leq d$.
Then we may define
\[ \tilde{q} := \prod_{k=1}^{d} p_{i_k,j_k}(a_k,b_k)\prod_{k=1}^{d} p_{s_k,t_k}(0,0) - \prod_{k=1}^{d} p_{s_k,t_k}(c_k,d_k)\prod_{k=1}^{d} p_{i_k,j_k}(0,0).
\]
Also, one can modifiy $\tilde{q}$ by taking $k$ in a subset of $\{1,\dots,d\}$, as we may not need every $k$ from $1$ to $d$ to make $q$ multi-homogeneous.
We call each such $\tilde{q}$ a {\bf lifting} of $q$, including the case when $q$ is lifted using less then $d$ variables in each term. \\
(Note: in addition, if not all of $(a_k,b_k)$ and $(c_k,d_k)$ are $(1,1)$, we may lift $q$ by $p_{i,j}(1,1)$
instead of $p_{i,j}(0,0)$ as we just described in $\tilde{q}$.)

It is clear that all of these lifts are in the toric ideal of the model. It is not clear that these are sufficient
to generate it (or give its Gr\"obner basis).  In particular, it seems that there is another kind of lifting that needs to
be included to generate the ideal of the model. Essentially, it involves overlapping two (or more!) minimal generators of the ideal $I_{\A_n}$ in a special way.

For example,  
consider the binomial 
\begin{align*}
	 p_{1,2}(1,0)p_{1,3}(0,1)p_{1,4}(0,1)p_{2,3}(1,1)- p_{1,2}(0,1)p_{1,3}(1,1)p_{2,3}(0,1)p_{2,4}(0,1) ,
\end{align*}
which is in the ideal for the model on $n\geq 4$ nodes.  It is not of the form $\tilde{q}$, that is, it is not lifted in the nice way described above.  However, it can still
be obtained from binomials on the graph on $3$ nodes.  Let 
\begin{align*}
	f^+&=p_{1,2}(1,0)p_{1,3}(0,1)p_{2,3}(1,0) \\
	f^- &=p_{1,2}(0,1)p_{1,3}(1,0)p_{2,3}(0,1)
\end{align*}
and 
\begin{align*}
	g^+ &=p_{1,4}(0,1)p_{2,3}(0,1)\\
	g^-  &=p_{1,3}(0,1)p_{2,4}(0,1).
\end{align*}
Note that $f:= f^+-f^-$ and $g:= g^+-g^-$ are in the simplified model ideal for $n\geq 3$ nodes. 
If we use the fact that in the graph the edges from node $1$ to node $3$ and from node $3$ to node $1$ combine to a
double-ended edge  between $1$ and $3$ to rewrite the binomial
$$
	f\boxtimes g := f^+ g^+  -  f^- g^-
$$
then we obtain precisely 
\begin{align*}
	 p_{1,2}(1,0)p_{1,3}(0,1)p_{1,4}(0,1)p_{2,3}(1,1)- p_{1,2}(0,1)p_{1,3}(1,1)p_{2,3}(0,1)p_{2,4}(0,1) .
\end{align*}
We will call such an operation ($f\boxtimes g$) an {\bf overlap} of two binomials, since it corresponds to overlapping the edges of the graphical representations of $f^+-f^-$ and $g^+-g^-$. 
	Take a binomial $f\boxtimes g$ in the ideal of the model $I_{\E_n}$.  Note that it may happen that 
	  neither $f$ nor $g$ are in $I_{\E_n}$. But in terms of 
	\emph{moves}, $f\boxtimes g$ is equivalent to performing two successive moves: the one defined by $f$, and the one defined by $g$.	
	In particular, binomial overlaps give rise to consecutive Markov moves which respect multi-homogeneity.  

\begin{rmk}\rm
	Note that $Q$ appears only in the decomposition for the ideal $I_{\tilde\E_n}$, and not $I_{\tilde\Z_n}$. But for example the binomial
	\[
		p_{12}(1,1)p_{34}(1,1)p_{23}(0,0)p_{14}(0,0)  -  p_{12}(0,0)p_{34}(0,0)p_{23}(1,1)p_{14}(1,1)
	\]
	is a homogenization of a generator of $Q$, and it lives in the ideal $I_{\tilde\Z_4}$. Homogenization by $p_{ij}(0,0)$ does not affect the 
	\emph{move} itself.
\end{rmk}

\begin{proof}[Proof of Theorem \ref{thm:essentialMarkovForModel}]
	Clearly, all moves from $I_{\A_n}$ and $Q$ can be homogenized by lifting:
	that is, if $q\in I_{\A_n}$, then $\tilde{q}\in I_{\E_n}$.

	The proof relies on a simple, yet crucial, observation that by definition, the ideal of the model is contained in the ideal of the 
	simplified model; e.g. 
	\[ I_{\E_n} \subset I_{\tilde\E_n}.
	\]

	Let $f\in I_{\E_n}$.  Then $f\in I_{\tilde\E_n}$.  If $f$ is in the Graver basis of $I_{\tilde\E_n}$, then we are conclude by 
	  Theorem \ref{thm:Gbasis-of-A(edgeRho)-using-graphs}.
		Alternately, assume $f$ is not in the Graver basis of $I_{\tilde\E_n}$. Then there exists a binomial $p\in I_{\tilde\E_n}$ such that $p^+$ divides $f^+$
	and $p^-$ divides $f^-$. 
	Equivalently, we can write $f$ as $f=p\boxtimes \tilde f$ for $\tilde f$ defined appropriately (e.g. $\tilde f^+ := f^+ / p^+$).
	Since we may assume that $p$ is primitive, we can keep decomposing $\tilde f$ until we obtain an overlap of $k$ primitive binomials from 
	the ideal of the simplified model.  
	
	  Note also that we may assume that in the end we are not in the case where $\tilde f^+ = \tilde f^-$. Indeed,
	if $f$ is a multiple of another binomial in the ideal, say, $f=\tilde f^+ g^+ - \tilde f^+ g^-$, then $g^+-g^-$ is also in the ideal. 
	  In terms of moves, multiples do not contribute anything: they instruct us to remove and add the same edge.
	  
	Replacing $\E_n$ by $\Z_n$ does not change the above argument.
\end{proof}

Using these two constructions of lifting and overlapping, we make the following conjecture:
\begin{conj}
    Minimal Markov (Gr\"obner) bases for the $p_1$ models can be obtained from Markov (Gr\"obner) bases of $I_{\A_n}$
    by repeated \emph{lifting} and \emph{overlapping} of the binomials in the minimal Markov bases of various $(n-1)$-node subnetworks.
\end{conj}

We close this section by remarking that these various liftings imply a (possibly non-sharp) degree bound on the Markov and Gr\"obner bases for the model.
For example, each lifting of the first stype will add at most ${n\choose 2}$ edges to each monomial of $q$ thus increasing the degree by at most ${n\choose 2}$ from
the degree needed for the simplified model, while overlapping $k$ binomials will allow the degrees of generators to increase $k$ times.  

Note that we have already seen lifts and overlaps in Examples \ref{eg:overlap} and \ref{eg:overlap-and-lift}.

\smallskip

\section{The marginal Polytope and the MLE}\label{sec:mle}

In this section we describe the geometric properties of $p_1$ models and we investigate conditions for existence of the MLE.

\subsection{Geometric Properties of $p_1$ Models}

As we explained in section \ref{sec:p1.model}, the statistical model specified by a $p_1$ model with design matrix $A$ is the set $M_A$ of all vectors satisfying the Holland and Leinhardt equations \eqref{eq:hl.eqs}. 
Letting $V_A \subset \mathbb{R}^{2n(n-1)}$ be the toric variety corresponding to the toric ideal $I_A$ (see section \ref{section:algebra-toric-ideal-of-p1}), it is readily verified that the closure of $M_A$ is precisely the set
\[
S_A := V_A \cap D_n,
\] 
where $D_n =  \left\{ \left(p_{1,2}, p_{1,3}\ldots,p_{n-1,n} \right) \colon p_{i,j} \in \Delta_3, \forall i \neq j \right\}$, with $\Delta_3$ the standard $3$-simplex.  Any point in $S_A$ is comprised of ${n \choose 2}$ probability distributions over the four possible dyad configurations $\{(0,0)$, $(1,0)$, $(0,1)$, $(1,1)\}$, one for each pair of nodes in the network. However, in addition to all strictly positive dyadic probabilities distributions satisfying \eqref{eq:hl.eqs}, $S_A$ contains also probability distributions with zero coordinates that are obtained as pointwise limits of points in $M_A$. 

The {\it marginal polytope} of the $p_1$ model with design matrix $A$ is defined as
\[
P_A := \{ t \colon t = A p, p \in D_n\}.
\]
As we will see, $P_A$ plays a crucial role for deciding whether the MLE exists given an observed network $x$.
For our purposes, it is convenient to represent $P_A$ as a Minkowski sum. To see this, decompose the matrix $A$ as
\[
A = [A_{1,2} A_{1,3} \ldots A_{n-1,n}], 
\]
where each $A_{i,j}$ is the submatrix of $A$ corresponding to the four indeterminates
\[
\{ p_{i,j}(0,0),p_{i,j}(1,0), p_{i,j}(0,1),p_{i,j}(1,1) \}
\] 
for the $(i,j)$-dyad. Then, denoting by $\sum$ the Minkowski sum of polytopes, it follows that 
\[
P_A = \sum_{i < j} \mathrm{conv}(A_{i,j}).
\]
We will concern ourselves with the boundary of $P_A$, which we will handle using the following theorem, due to Gritzman and Sturmfels \cite{gritzman.sturmfels:1993}. For a polytope $P \subset \mathbb{R}^d$ and a vector $c \in \mathbb{R}^d$, we write
\[
S(P ; c) := \{ x \colon x^\top c \geq y^\top c, \forall y \in P\}
\]
for the set of maximizers $x$ of the inner product $c^\top x$ over $P$.
\begin{theorem}[\cite{gritzman.sturmfels:1993}]\label{thm:GS:93}
Let $P_1, P_2, \ldots, P_k$ be polytopes in $\mathbb{R}^d$ and let $P = P_1 + P_2 +  \ldots + P_k$. 
A nonempty subset $F$ of $P$ is a face of $P$ if and only if 
$F = F_1 + F_2 + \ldots + F_k$ for some face $F_i$ of $P_i$ such that there exists $c \in \mathbb{R}^d$  with $F_i = S(P_i ; c)$ for all $i$. Furthermore, the decomposition 
$F = F_1 + F_2 + \ldots + F_k$ of any nonempty face $F$ is unique. 
\end{theorem}

The marginal polytope $P_A$ and the set $S_A$ are closely related. Indeed, as shown in \cite{morton:2008}, the sets $S_A$ and $P_A$ are homeomorphic:
\begin{theorem}[\cite{morton:2008}]\label{thm:moment.map}
The map $\mu \colon S_A \rightarrow P_A$ given by
\[
p \mapsto A p
\]
is a homemorphism.
\end{theorem}

 This result is a generalization of the moment map theorem for toric varieties defined by homogeneous toric ideals (see \cite{fulton:1993} and \cite{ewald:1996}) to the  multi-homogeneous case.  In particular, it implies that $P_A = \{ t \colon t = A p, p \in S_A\}$.

\subsection{Existence of the MLE in $p_1$ Models}

We focus on two important problems related to maximum likelihood estimation in $p_1$ models that have been traditionally overlooked in both theory and practice. Our results should be compared in particular, with the ones given in \cite{haberman:77}. The geometric properties of $S_A$ and $P_A$ established above are fundamental tools of our analysis.
 
We begin 
a result justifying the name for the marginal polytope.
\begin{lm}\label{lm:marg.pol}
The polytope $P_A$ is the convex hull of the set of all possible observable marginals; in symbols:
\[
P_A = \mathrm{conv}(\{ t \colon t = Ax, x \in \mathcal{X}_n \}).
\]
\end{lm}

\begin{proof}
If $t = Ax$ for some  $x \in \mathcal{X}_n$, then, by definition, $t$ is in the Minkowski sum of $\mathrm{conv}(A_{i,j})$'s for all $i<j$, yielding that $\mathrm{conv}(\{ t \colon t = Ax, x \in \mathcal{X}_n \})$ is a subset of $P_A$. Conversely, by theorem \ref{thm:GS:93}, any extreme point $v$ of $P_A$ can be written as $v = v_{1,2} + v_{1,3} + \ldots + v_{n-1,n}$, where each $v_{i,j}$ is an extreme point of $\mathrm{conv}(A_{i,j})$ such that $v_{i,j} = S(\mathrm{conv}(A_{i,j});c)$ for some vector $c$ and all $i < j$. Since, for every $i \neq j$, the columns of $A_{i,j}$ are affinely independent, they are also the extreme points of $\mathrm{conv}(A_{i,j})$. Therefore, the extreme points of $P_A$ are a subset of $\{ t \colon t = Ax, x \in \mathcal{X}_n \}$, which implies that $P_A$ is a subset of $\mathrm{conv}(\{ t \colon t = Ax, x \in \mathcal{X}_n \})$. 
\end{proof}

Let $x$ denote the observed network. From standard statistical theory of exponential families (see, e.g., \cite{brown:86} and \cite{lau:96}), the MLE $\hat{p}$ exists if and only if the vector of margins $t = Ax$ belongs to the relative interior of $P_A$, denoted by $\mathrm{ri}(P_A)$. Furthermore, it is always unique and satisfies the {\it moment equations,} which we formulate using the moment map of Theorem \ref{thm:moment.map} as 
\begin{equation}\label{eq:mle}
\hat{p} = \mu^{-1}(t).
\end{equation}

When the MLE does  not exists, the likelihood function still achieves its supremum at a unique point, which also satisfies the moment equations \eqref{eq:mle}, and is called the {\it extended MLE}. The extended MLE is always a point on the boundary of $S_A$ and, therefore, contains zero coordinates. Although not easily interpretable in terms of the model parameters of the Holland and Leinhardt equations \eqref{eq:hl.eqs}, the extended MLE still provides a perfectly valid probabilistic representation of a network, the only difference being that certain network  configurations have a zero probability of occurring.

Here we are concerned with the following two problems that are essential for computing the MLE and extended MLE:
\begin{enumerate}
\item  decide whether the MLE exists, i.e. whether observed vector of margins $t$ belongs to $\mathrm{ri}(P_A)$, the relative interior of the marginal polytope; 
\item compute  $\mathrm{supp}(\hat{p})$, the support of $\hat{p}$, where $\mathrm{supp}(x) = \{ i \colon x_i \neq 0\}$. Clearly, this second task is nontrivial only when $t$ is a point on the relative boundary of $P_A$, denoted by $\mathrm{rb}(P_A)$, and we are interested in the extended MLE.
\end{enumerate}

Both problems require handling the geometric and combinatorial properties of the marginal polytope $P_A$ and of its faces, and to relate any point in $\mathrm{rb}(P_A)$ to the boundary of $S_A$. In general, this is challenging, as the combinatorial properties of Minkowki sums tend to be high. Fortunately, we can simplify these tasks by resorting to a larger polyhedron that is simpler to work with.

Let $C_A = \mathrm{cone}(A)$ be the {\it marginal cone} (see \cite{mle1:06}), which is easily seen to the convex hull of all the possible observable sufficient statistics if there were no constraint on the number of possible observations per dyad. Notice that, since the columns of $A$ are affinely independent, they define the extreme rays of $C_A$. Following \cite{gms:06}, a set $\mathcal{F} \subseteq \{1,2,\ldots,2n(n-1) \}$ is called a \emph{facial set} of $C_A$ if there exists a $c$ such that 
\[
c^\top a_i = 0, \;\;\; \forall i \in \mathcal{F} \quad \text{and} \quad c^\top a_i < 0, \;\;\;  \forall i \not \in \mathcal{F},
\]
where $a_i$ indicates the $i$-th column of $A$.  It follows that $F$ is face of $C_A$ if and only if $F = \mathrm{cone}(\{a_i \colon i \in \mathcal{F}\})$, for some facial set $\mathcal{F}$ of $C_A$, and that there is a one-to-one correspondence between the facial sets and the faces of $C_A$.

In our main result of this section, we show that the existence of the MLE can be determined using the simpler set $C_A$ and that the supports of the points on the boundary of $S_A$ are facial sets of $C_A$. 

\begin{theorem}\label{thm:mle}
Let $t \in P_A$. Then, $t \in \mathrm{ri}(P_A)$ if and only if $t \in \mathrm{ri}(C_A)$.
Furthermore, for every face $G$ of $P_A$ there exists one facial set $\mathcal{F}$ of $C_A$ such that $\mathcal{F} = \mathrm{supp}(p)$, where $p = \mu^{-1}(t)$, for each $t \in \mathrm{ri}(G)$.
\end{theorem}

\begin{proof}
By Lemma \ref{lm:marg.pol}, $P_A \subset C_A$. Thus, since both $P_A$ and $C_A$ are closed, $t \in \mathrm{ri}(P_A)$ implies $t \in \mathrm{ri}(C_A)$. 
For the converse statement, suppose $t$ belongs to a proper face $G$ of $P_A$. Then, by Theorem \ref{thm:GS:93},
\begin{equation}\label{eq:t}
t = t_{1,2} + t_{1,3} + \ldots + t_{n-1,n}
\end{equation}
where $t_{i,j}$ belongs to a face of $\mathrm{conv}(A_{i,j})$, for each $i < j$. We now point out two crucial features of $P_A$ that can be readily checked:
\begin{enumerate}
\item[(i)] the first ${n \choose 2}$ coordinates of any point in $P_A$ are all equal to $1$;
\item[(ii)] the first ${n \choose 2}$ coordinates of each $t_{i,j}$ are all zeros except one, which takes on the value  $1$.
\end{enumerate}
Suppose now that $t \in \mathrm{ri}(C_A)$. Because of (i) and (ii), there exists a point in $x \in D_n$ with strictly positive entries such that $t = A x$. In turn, this implies that
\[
t = t'_{1,2} + t'_{1,3} + \ldots + t'_{n-1,n},
\]
where $t'_{i,j} \in \mathrm{ri}(\mathrm{conv}(A_{i,j}))$ for each $i < j$, which contradicts \eqref{eq:t}. Thus, $t \in \mathrm{rb}(C_A)$.
To prove the second claim, notice that, because of the uniqueness of the decomposition of Theorem \ref{thm:GS:93}, our arguments further yield that, for every proper face $G$ of $P_A$, there exists one face $F$ of $C_A$ such that $\mathrm{ri}(G) \subseteq \mathrm{ri}(F)$. Consequently, to every face $G$ of $P_A$  corresponds one facial set $\mathcal{F}$ of $C_A$ such that, if $t \in \mathrm{ri}(G)$, then
\begin{equation}\label{eq:p}
t = Ap
\end{equation}
for some $p \in D_n$ with $\mathrm{supp}(p) = \mathcal{F}$. Since columns of $A$ are affinely independent, this $p$ is unique.  By Theorem \ref{thm:moment.map}, there exists a unique point $p' = \mu^{-1}(t) \in S_A$ satisfying \eqref{eq:p}, so that $\mathrm{supp}(p') \subseteq \mathcal{F}$. Inspection of the proof of the same theorem yields that, in fact, $\mathrm{supp}(p')= \mathcal{F}$ (see also the appendix in  \cite{gms:06}), so that $p'=p$. 
\end{proof}

From the algorithmic standpoint, the previous theorem is rather useful, as it allows us to work directly with $C_A$, whose face lattice is much simpler (for example, while we know the extreme rays of $C_A$, it is highly nontrivial to find the vertices of $P_A$ among $4^{n(n-1)}$ vectors of observable sufficient statistics).
Algorithms for deciding whether $t \in \mathrm{ri}(C_A)$ and for finding the facial set corresponding to the face $F$ of $C_A$ such that $t \in \mathrm{ri}(F)$ are presented in \cite{mle1:06} and \cite{mle2}. 

We conclude this section with a statistical remark. Using the terminology of log-linear modeling  theory (see, e.g., \cite{bfh:07}), $p_1$ models are log-linear models arising from a  {\it product-Multinomial sampling scheme.} This additional constraint is what makes computing the Markov bases and dealing with $P_A$ particularly complicated in comparison with the same tasks under Poisson or Multinomial sampling scheme. In these simpler sampling settings, the toric ideal $I_A$ is homogenous, all Markov moves are applicable (provided that their degree is smaller than $\sum_i x_i$ for the Multinomial scheme), and the model $S_A$ is homoemorphic to $C_A$ for the Poisson scheme, and to $\mathrm{conv}(A)$ for the Multinomial scheme.

\subsection{Computations}
We now briefly summarize some of the computations we carried out involving $P_A$ and $C_A$, based on {\tt polymake} (see \cite{polymake}), {\tt minksum} (see \cite{minksum}) and our implementation of the iterative proportional scaling algorithms for fitting $p_1$ models described in \cite{holl:lein:1981}. Though this is a rather incomplete and limited report, it does  provide some indications of the complexity of $p_1$ models and of the subtleties of the  maximum likelihood estimation problem. 

In Table \ref{tab:C_A} we report the number of facets, dimensions and ambient dimensions of the cones $C_A$ for different values of $n$ and for the three specification of the reciprocity parameters $\rho_{i,j}$ we consider here. It is immediate to see how large the combinatorial complexity of $C_A$ is, even for network of really small dimension.

\begin{table}\label{tab:C_A}
\begin{tabular}[t]{|c|ccc|ccc|ccc|}
\hline
$n$ & \multicolumn{3}{|c|}{$\rho_{i,j}=0$} & \multicolumn{3}{|c|}{$\rho_{i,j}=\rho$} & \multicolumn{3}{|c|}{$\rho = \rho_i + \rho_j$}\\
\hline
& Facets & Dim. & Ambient Dim. & Facets &  Dim. & Ambient Dim. & Facets &  Dim. & Ambient Dim.\\
\hline
3 & 30 & 7 & 9 & 56 & 8 & 10 & 56 & 8 & 13\\
4 & 132 & 12& 14& 348 & 13 & 15 & 348 & 13 & 19\\
5 & 660 & 18 & 20 & 3032 & 19 & 21 & 3032 & 19 & 26\\
6 & 3181 & 25 & 27 & 94337 & 26 & 28 & 94337 & 26 & 34\\
\hline
\end{tabular}
\caption{Number of facets, dimensions and ambient dimensions of the the cones $C_A$ for different specifications of the $p_1$ model and different network sizes.}
\end{table}

Another point of interest is the assessment of how often the existence of the MLE  arises. In fact, because of the product Multinomial sampling constraint, nonexistence of the MLE is quite severe, especially for smaller networks. However, in light of the following results, we are led to conjecture that the larger the network the higher the chance that the MLE t exists.   
Below we report our finding:
\begin{enumerate}
\item  $n=3$.\\
The sample space consists of $4^3 = 64$ possible networks. when $\rho_{i,j} = 0$, there are  $63$ different observable sufficient statistics, only one of which belongs to $\mathrm{ri}(P_A)$. Thus, only one of the $63$ observable sufficient statistics leads to the existence of the  MLE. The corresponding fiber contains only the following two networks:
\begin{center}
\begin{verbatim}
0 0 1 0 0 1 0 0 0 0 1 0 
0 1 0 0 0 0 1 0 0 1 0 0 
\end{verbatim}
\end{center}
and the associated MLE is the $12$-dimensional vector with entries all equal to $0.25$.
Incidentally, the marginal polytope $P_A$ has $62$ vertices and $30$ facets.
When $\rho_{i,j} = \rho \neq 0$ or $\rho_{i,j} = \rho_i + \rho_j$ the MLE never exists. 
\item  $n=4$. \\
The sample space contains $4096$ observable networks. If $\rho_{i,j}=0$, there are $2,656$ different observable sufficient statistics, only $64$ of which yield existent MLEs. Overall, out of the $4096$ possible networks, only $426$ have MLEs.
When $\rho_{i,j}=\rho \neq 0$, there are $3,150$ different observable sufficient statistics, only $48$ of which yield existent MLEs. Overall, out of the $4096$ possible networks, only $96$ have MLEs. When $\rho_{i,j} = \rho_i + \rho_j$, $3150$ different observable sufficient statistics.
\item  $n=5$. \\
The sample space consists of  $4^{10} = 1,048,576$ different networks. If $\rho_{i,j}=0$, there are $225,025$ different sufficient statistics, while if $\rho_{i,j} = \rho \neq 0$ or $\rho_{i,j} = \rho_i + \rho_j$, this number is $349,500$. Even though $n$ is very small, this case is already computationally prohibitive!
\end{enumerate}

\subsubsection{The case $\rho = 0$}
When $\rho = 0$, nonexistence of the MLE can be more easily detected. Even though this is the simplest of $p_1$ models, the observations below apply to more complex $p_1$ models as well, since nonexistence of the MLE for he case $\rho = 0$ implies nonexistence of the MLE in all the other cases.

Only for this case, it is more convenient to switch to the parametrization originally used in \cite{holl:lein:1981}, and describe networks using incidence matrices. Thus, each network on $n$ nodes can be represented as a $n \times n$ matrix with $0/1$ off-diagonal entries, where the $(i,j)$ entry is $1$ is there is an edge from $i$ to $j$ and $0$ otherwise. 
In this parametrization, the sufficient statistics for the parameters $\{ \alpha_i, i=1,\ldots,n\}$ and $\{ \beta_j, j=1,\ldots,n\}$ are the row and column sums, respectively. Just like with the other $p_1$ models, the minimal sufficient statistic for the density parameter $\theta$ is the total number of edges, which is a linear function of the sufficient statistics for the other parameter; because of this, it can be ignored. 

We remark that, while this parametrization is more intuitive and parsimonious than the one we chose, whenever $\rho \neq 0$, the sufficient statistics for the reciprocity parameter are not linear functions of the observed network. As a result, the Holland and Leinhardt parametrization does not lead to toric or log-linear models.

When $\rho = 0$, there are three cases iwhere the MLE does not exist. The first two are of capable of immediate verification. 
\begin{enumerate}
\item If a row or column sum is equal to $n-1$, this implies that the either $\hat{p}_{ij}(0,1) = \hat{p}_{ij}(0,0) = 0$ for some $i$, or $\hat{p}_{ij}(1,0)  = \hat{p}_{ij}(0,0) = 0$ for some $j$.
\item If a row or column sum is equal to $0$, this implies that the either $\hat{p}_{ij}(1,0) = \hat{p}_{ij}(1,1) = 0$ for some $i$, or $\hat{p}_{ij}(0,1) = \hat{p}_{ij}(1,1) = 0$ for some $j$.
\item The last case is subtler. From the theory of exponential families, it follows that the MLE $\hat{p}$ satisfies the moment equations, namely the row and column sums of $\hat{p}$ match the corresponding row and column sums of the observed network $x$. Thus, the MLE does not exists whenever this constrained cannot be satisfied by any strictly positive vector. As a result, for $n=3$, besides the two obvious cases indicated above, the MLE does not exists if the following patterns of zeros are observed:
\[
\left[ 
\begin{array}{ccc}
\times & 0 & $\;$\\
0 & \times & $\;$\\
$\;$ & $\;$ & \times
\end{array}
\right], \quad 
\left[ 
\begin{array}{ccc}
\times & $\;$ & 0\\
$\;$ & \times & $\;$\\
0 & $\;$ & \times
\end{array}
\right], \quad
\left[ 
\begin{array}{ccc}
\times & $\;$ & $\;$\\
$\;$ & \times & 0\\
$\;$ 0 & & \times
\end{array}
\right].
\]

When $n=4$, there are 4 patterns of zeros leading to a nonexistent MLE, even though the margins can be positive and smaller than $3$:
\[
\left[ 
\begin{array}{cccc}
\times & 0 & $\;$ & 0\\
0 & \times & $\;$ & 0\\
$\;$ & $\;$ & \times & $\;$\\
0 & 0 & $\;$ & \times\\
\end{array}
\right], \quad 
\left[ 
\begin{array}{cccc}
\times & 0 & 0 & $\;$\\
0 & \times & 0 & $\;$\\
0 & 0 & \times & $\;$\\
$\;$ & $\;$ & $\;$ & \times\\
\end{array}
\right],\quad 
\left[ 
\begin{array}{cccc}
\times & $\;$ & 0 & 0\\
$\;$ & \times & $\;$ & $\;$\\
0 & $\;$ & \times & 0\\
0 & $\;$ & 0 & \times\\
\end{array}
\right], \quad
\left[ 
\begin{array}{cccc}
\times & $\;$ & $\;$ & $\;$\\
$\;$ & \times & 0 & 0\\
$\;$ & 0 & \times & 0\\
$\;$ & 0 & 0 & \times\\
\end{array}
\right].
\]

These patterns have been found using {\tt polymake} in the following way. Represent the off-diagonal entries of the $n \times n$ incidence matrix matrix $x$ as a $n(n-1)$-dimensional vector, with the coordinates ordered, say, 
 lexicographically. Let $B$ be a $2n \times n(n-1)$ matrix such that $Bx$ is a vector of the row and column sums of the original matrix. For example, when $n=3$, 
 \begin{align*}
   B = \begin{bmatrix} 
     1  &  0    &  1  &  0    &  0  &  0 \\
     0  &  1    &  0  &  0    &  1  &  0  \\
     0  &  0    &  0  &  1   &  0  &  1  \\
     0  &  1    &  0  &  1    &  0  &  0 \\
     1  &  0    &  0  &  0    &  0  &  1  \\
     0  &  0    &  1  &  0    &  1  &  0  
  \end{bmatrix},
\end{align*}
where its columns correspond to the coordinates $(1,2)$, $(1,3)$, $(2,1)$, $(2,3)$, $(3,1)$ and $(3,2)$.
As the columns of $B$ are affinely independent,  in order to determine the patterns of zeros leading to a nonexistent MLE, it is sufficient to look at the faces of the cone $\mathrm{cone}(B)$. Based on our computations, which we carried out for networks of size up to $n=10$, for a network on $n$ nodes, there are $3n$ facets, $2n$ of which correspond to patterns of zeros leading to a zero row or column margin, and the remaining $n$ to patterns of zeros which cause nonexistent MLE without inducing zero margins, like the ones presented above. 
\end{enumerate} 

Finally, we note that the matrix $B$ described can also be obtained as a submatrix of the design matrix $\mathcal{Z}_n$ after removing the rows corresponding to the parameters $\lambda_{ij}$'s and the columns corresponding to the indeterminates $p_{ij}(1,1)$, since those columns are simply the sum of the two columns corresponding to the indeterminates $p_{ij}(1,0)$ and $p_{ij}(0,1)$, for all $i \neq j$. 
Referring back to Section \ref{sec:markov}, we notice that we have seen the matrix $B$ before.  In fact, $B=\A_n$, what we referred to as the \emph{common submatrix} of each of the three version of the $p_1$ model. 
In Theorems \ref{thm:simplified-noRho-decomposition} and \ref{thm:simplified-EdgeRho-decomposition}, we have seen that the Markov basis of the toric ideal of this matrix essentially determines the Markov bases for the simplified models. 
More importantly, in Theorem \ref{thm:essentialMarkovForModel}, it plays a crucial role in determining the essential Markov moves for the model.

\smallskip

\section{Discussion}

We close with a  discussion   about the relationship between our parametrization and the log-linear parametrization suggested by \cite{fien:wass:1981,fien:wass:1981a}.
Fienberg and Wasserman's parametrization of the $p_1$ model encodes it as a $n^2 \times 2 \times 2$ contingency table, where the first variable corresponds to a dyad and the second and third represent the four dyadic configurations. Thus, a network is represented as a point $x$ in $\mathbb{R}^{4n^2}$ with 0/1 entries. 
This log-linear parametrization is clearly highly  redundant, as, besides the Multinomial constraints on each of the $n^2$ dyads, there are additional symmetric constraints of the form $x_{i,j,k,l} = x_{j,i,k,l}$, for all $i,j \in \{1,\ldots,n\}$ and $k,l \in \{0,1\}$. 
Although, as shown in \cite{fien:meye:wass:1985}, these redundancies can be convenient when computing the MLE, they are highly undesirable for finding Markov bases. Indeed, the toric ideal corresponding to the this parametrization has $4n^2$ indeterminates while our parametrization only contains $2n(n-1)$.
For example, when $n=5$, this means the toric ideal lives in the polynomial ring with $100$ indeterminates, instead of $40$ that we have.  
In addition, the number of generators explodes combinatorially: for the case of constant reciprocation, $\rho_{ij}=\rho$, the ideal of the network on $n=3$ nodes has $107$ minimal generators, and the one of the $4$-node network has $80,610$.  
 The case when $n=5$ is hopeless: even {\tt 4ti2}, the fastest software available to compute generating sets of toric ideals, cannot handle it. 
One wonders what all of these extra generators are! First of all, note that 
this contingency-table representation is highly redundant.  Most of the Markov basis elements are
inapplicable because of the product multinomial constraints and the symmetry constraints. 
Finally, the many symmetries in the table imply many highly non-trivial constraints that have to be accounted for when eliminating non applicable moves. 
We were able to analyze the $n=4$ case and reduce all of the $80610$ moves to the ones we get using the design matrices $\C_n$, but the effort was nontrivial. 
Therefore, at least from the point of view of studying Markov bases, the parametrization we are using  in the present paper is preferable. 

Our hope is that this article motivates a deeper study of the algebraic statistical challenges for $p_1$ models and their extensions.


\section*{Acknowledgements}
The authors are grateful to Di Liu for sharing her code for fitting $p_1$ models.


\end{document}